\documentclass[twoside,leqno,11pt]{article}

\usepackage{a4wide}

\usepackage{amsmath}

\usepackage{color}
\usepackage{amsmath,amssymb,enumerate}
\usepackage{pstricks}

\usepackage{epsfig}

\usepackage{hyperref}
\usepackage{dsfont}
\usepackage[all]{xy}

\usepackage{pgfplots}
\pgfplotsset{compat=1.15}
\usepackage{mathrsfs}
\usetikzlibrary{arrows}
\usetikzlibrary{patterns}
\usepackage{xcolor}
\usepackage{tikz}
\usetikzlibrary{calc,trees,positioning,fit,shapes,calc}

\usepackage{graphics}
\usepackage{pst-3d,float}
\newtheorem{theorem}{Theorem}[section]

\newtheorem{corollary}[theorem]{Corollary}
\newtheorem{definition}[theorem]{Definition}
\newtheorem{lemma}[theorem]{Lemma}
\newtheorem{proposition}[theorem]{Proposition}
\newtheorem{remark}[theorem]{Remark}
\newenvironment{proof}{\medskip\noindent{\bf Proof.}\;}{\null\hfill $\Box$\par\medskip }

\newcommand{\C}{\ensuremath{\mathcal C}}
\newcommand{\I}{\ensuremath{\mathcal I}}
\newcommand{\K}{\ensuremath{\mathcal K}}
\newcommand{\qed}{\hfill\ensuremath{\square}}

\title{ Lacunary wavelet series on Cantor sets}
\author{ C\'eline Esser$^a$\footnote{Corresponding author. Email:
    celine.esser@uliege.be} and B\'eatrice Vedel$^b$\\\\
$^a$ Universit\'e de Li\`ege,  All\'ee de la D\'ecouverte 12, B-4000
Li\`ege, Belgium \\
$^b$Univ Bretagne Sud, CNRS UMR 6205, LMBA, F-56000 Vannes, France}         

\date{ }

\newsavebox{\fmbox}
\newenvironment{fmpage}[1]
 {\begin{lrbox}{\fmbox}\begin{minipage}{#1}}
 {\end{minipage}\end{lrbox}\fbox{\usebox{\fmbox}}}

\def \Z {\mathbb{Z}}
\def \N {\mathbb{N}}
\def \R {\mathbb{R}}
\def \C {\mathcal{C}}

\def \PP {{\mathbb{P}}}
\def \EE {{\mathbb{E}}}

\def \G {\Gamma}

\begin{document}
\maketitle

\begin{abstract}
We provide a multifractal analysis of lacunary wavelet series on Cantor sets. By introducing a desynchronization between the scales of the wavelets and the scales of the steps of the construction of the Cantor set, we obtain random processes that do not satisfy multifractal formalisms based on the Legendre transform and on the large deviation of wavelet leaders. Subsequently, we show how the computation of the leader large deviation spectra of ``lacunarized'' versions of such wavelet series can detect the failure in the multifractal formalism using only numerical quantities.
\end{abstract}

\noindent  {\bf Keywords :} Multifractal Analysis, Multifractal Formalism, Lacunary wavelet series, Large deviation spectrum,   Cantor sets\\

\noindent {\bf 2010 Mathematics Subject Classification : } 42C40, 28A80, 26A16, 60G17\\

\section{Introduction}

This paper  focuses on lacunary wavelet series defined on fractal sets. In the first part, we perform a multifractal analysis of random lacunary wavelet series defined on the steps of the construction of a Cantor set. The aim is twofold: to determine the multifractal spectrum and to evaluate the validity of the formalism for such simple processes. This leads to the second part of the paper, which focuses on a criterium to detect failures in the numerical estimation of the multifractal spectrum based on computable quantities for lacunary wavelet series on fractal sets.

The constructions presented in this paper give a natural generalisation of the random lacunary wavelet series defined on $[0,1]$  in the seminal paper \cite{Jaffard:00b}, where their multifractal analysis is performed. Let us recall it.

We use an orthonormal wavelet basis on $\R$, given by two functions
$\varphi$ and  $\psi$ in the Schwartz class with infinite numbers of vanishing moments and with the property that
the family
$$
\{ \varphi(\cdot -k) :\, k \in \Z\} \cup \{ 2^{\frac{j}{2}} \psi(2^j
\cdot -k): \, j \in \N, k \in \Z \}
$$
forms an orthonormal basis of $L^2(\R)$, see \cite{LM86}. Therefore, for all $f \in L^2(\R)$, we have the following decomposition 
$$
f= \sum_{k \in \Z} C_k \varphi(\cdot -k) +  \sum_{j \in \N} \sum_{k \in \Z} c_{j,k} \psi(2^j\cdot -k)
$$
where the wavelet coefficients of $f$ are defined by
$$
C_k = \int_{\R} f(x) \varphi(x-k) dx
$$
and
$$
c_{j,k} = 2^{j} \int_{\R} f(x) \psi(2^jx -k)dx.
$$
Note that we do not use the $L^2$ normalisation to avoid a rescaling
in the definition of the wavelet leaders, see Definition
\ref{def:leader} below. 
We introduce the following concise notations using dyadic intervals to index wavelets. For any $j \in \mathbb{N}$, $\Lambda_j$ denotes the set of all dyadic intervals in $[0, 1]$ of scale $j$. Furthermore, if $\lambda=\lambda_{j,k}= [k2^{-j}, (k+1)2^{-j}) \in \Lambda_j$, we write $c_{\lambda}=c_{j,k}$ and $\psi_{\lambda}= \psi_{j,k}=
\psi(2^j\cdot -k)$. 

Note that we could also work with regular wavelets such as the Daubechies wavelets (see \cite{Daubechies:92}), with an easy adaptation of the results as soon as we take a sufficient number of vanishing moments. 

Let us now consider two parameters $\alpha >0$ and $\eta \in (0,1)$. The Lacunary Wavelet Series (denoted by LWS) on $[0,1]$ of parameters $\alpha$ and $\eta$ is the process defined  by 
$$
F_{\alpha,\eta} = \sum_{j \in \N} \sum_{k=0}^{ 2^j - 1 } c_{j,k} \psi_{j,k}
\quad
\text{with} \quad 
c_{j,k} = 2^{-\alpha j} \xi_{j,k}
$$
where $(\xi_{j,k})_{j,k}$ denotes a sequence of independent random Bernoulli variables of parameter $2^{(\eta-1)j}$.  This setup ensures that, on average, at every scale $j$, there are $2^{\eta j}$ non-zero coefficients. The parameter $\eta$ characterizes therefore the lacunarity of the random series at each scale, while the parameter $\alpha$ is  directly related to its uniform H\"older regularity. 
These processes are very simple but turn out to have a rich behaviour in term of pointwise regularity. This can be observed through their multifractal analysis, a concept that allows to obtain a geometric description of the different pointwise regularities within a signal. Let us be more precise. 

In the following definition and in all the text, $\lfloor a \rfloor$ denotes the entire part of a real $a$.

\begin{definition}
Let $x_0 \in \R$ and $h>0$. A locally bounded function $f : \R \to \R$ belongs to $\C^{h}(x_0)$ if there exists $C>0$ and a polynomial $P_{x_0}$ with $\deg P_{x_{0}} < \lfloor h \rfloor$  such that 
$$
\vert f(x)-P_{x_0}(x) \vert \le C \vert x -x_0 \vert^{h}
$$
on a neighborhood of $x_0$. The \emph{pointwise H\"older exponent} of $f$ at $x_0$ is 
$$
h_f(x_0) = \sup \{ h\geq 0 : \, f \in \C^{h}(x_0)\}.
$$
The \emph{iso-H\"older sets} of $f$ are defined for every $h \in [0, +
\infty]$ by 
$$
I_f(h) = \{x_0 \in \R : \, h_f(x_0) = h\}.
$$
\end{definition}

For multifractal functions, whose pointwise H\"older exponent changes at each point, an
interesting information may be not to describe precisely each
iso-H\"older set but  rather to determine the Hausdorff dimension of the set.

\begin{definition}\label{def:spectrum}
The \emph{multifractal spectrum} $\mathcal{D}_{f}$ of a locally bounded function $f$ is the function 
$$
 {\cal{D}}_f : h \in  [0, + \infty] \mapsto \dim_{\mathcal{H}}I_f(h) 
$$
where $\dim_{\mathcal{H}}$ denotes the Hausdorff dimension. 
\end{definition}

In \cite{Jaffard:00b}, Jaffard derived the multifractal spectrum of any LWS on $[0,1]$. Furthermore, a direct corollary of Jaffard's results is the computation of the so-called \emph{leader large deviation spectrum} of an LWS. In essence, the leader large deviation spectrum $\rho_f$ of a function $f$ is defined such that, for every $h > 0$, there are approximately $2^{\rho_f(h)j}$ wavelet leaders of $f$ of order $2^{-hj}$, where the wavelet leaders can be seen as local maxima of wavelet coefficients. We refer to Section \ref{sec:formalisms} for a rigorous definition of the wavelet leaders and of the leader large deviation spectrum $\rho_f$. In general, it is expected that the functions $\mathcal{D}_f$ and $\rho_f$ coincide, though this is not always the case. If the equality $\mathcal{D}_f = \rho_f$ is true, it is said that \emph{the leader large deviation formalism holds}.

\begin{theorem}\cite{Jaffard:00b}
    Let $F_{\alpha, \eta}$ be a LWS of parameters $\alpha>0$ and $\eta \in(0,1)$. Almost surely, one has 
    $$
    \mathcal{D}_{F_{\alpha, \eta}}(h) = \rho_{F_{\alpha, \eta}}(h) = \begin{cases}  \frac{\eta}{\alpha}h & \text{ if } h \in \big[\alpha, \frac{\alpha}{\eta}\big], \\[1ex]
    - \infty & \text{ otherwise. }\end{cases}
    $$
\end{theorem}

The proof relies mainly on a transference theorem obtained on $[0,1]$. This result allows to estimate the Hausdorff dimension of limsup sets of contracted balls whose centers are ``well spread'' on $[0,1]$, see \cite{Jaffard:00b}. This transference principle can be extended to balls with centers in Cantor sets using the General mass transference principle of \cite{Beresn:06}. 
The latter principle enables us to study the multifractal properties of lacunary wavelet series on Cantor sets. We will study two different constructions.  The first one is a classical extension that yields multifractal processes adhering to the multifractal formalisms, see Section \ref{sec:formalisms} for the definitions of multifractal formalisms. Note that the extension of LWS in the context of Gibbs capacities with full support has been provided in \cite{BaSe20}. The second and more interesting one  involves a ``desynchronization'' between the natural scales of the Cantor set and the scales of the wavelets, introducing some redundancy in the wavelet coefficients. This produces examples of remarkably simple stochastic processes whose spectra deviate from the usual straight line that can be observed for the classical LWS, and rendering the formalisms invalid. This desynchronization, feasible only for processes defined on fractal sets rather than entire intervals, suggests that very rich behaviors can emerge when the self-similarity of the underlying set differs from that of the process. It's worth mentioning that a procedure creating apparent redundancy in the wavelet coefficients has been studied in parallel by Barral and Seuret in \cite{BaSe23}, starting from a Gibbs capacity with full support. The difference between their results and ours lies mainly in the support of the non-zero coefficients. The fractal nature of these supports allows us to desynchronize the localization of the non-zero coefficients from the natural corresponding step in the construction of the Cantor set. This results in two different self-similarity properties: one for the Cantor set and one for the wavelet coefficients.  This fractal nature also implies that any finite regularity is achieved on a set of dimension strictly smaller than $1$. \\

Let us delve into more detail about the construction and the results presented in this paper. We denote by $\C(r)$ the symmetric Cantor set with a dissection ratio  $r<\frac{1}{2}$, constructed iteratively as follows. We start with $\C_0 = [0, 1]$. Then, at each step $n$ in the construction, if we have constructed $\C_n$ as a union of $2^n$ closed intervals of length $r^n$, we remove the open middle interval of length $(1-2r)r^n$ from each of these intervals and we define $\C_{n+1}$ as the union of the remaining $2^{n+1}$ closed intervals of length $r^{n+1}$. Finally, we define the Cantor set $\C(r)$ by
\[
\C(r) = \bigcap_{n \in \N} \C_n.
\]
Next, we consider two parameters $\alpha >0$ and $\eta \in
(0,\gamma)$ where $$\gamma = \text{dim}_{\cal H}\, \C(r) = - \frac{1 }{\log_2 r}.$$  The model is given by the
random wavelet series $$F_{\alpha,\eta,r}=  \sum_{j \in \N}\sum_{k=0}^{2^j-1}c_{j,k}\psi_{j,k} \quad \text{with}\quad
c_{j,k}= \begin{cases}
2^{-\alpha j} \xi_{j,k} & \text{ if }  k \in \Gamma_{j}, \\
0 & \text{ otherwise,}
\end{cases}
$$
where
$$
\Gamma_{j} = \big\{ k \in \{0, \dots, 2^j-1\} : \lambda_{j,k} \subset  \C_{n_{j}}\big\}
\quad \text{ with } \quad n_{j}= \left\lfloor
\frac{-j}{\log_2 r}\right\rfloor =\left\lfloor
\gamma j\right\rfloor 
$$
 and 
where  $ (\xi_{j,k})_{j,k}$ denotes a sequence of independent random Bernoulli
variables of parameter $2^{(\eta-\gamma)j}$.  The random wavelet coefficients of scale $j$ are
located on the intervals of $\C_{n_{j}}$ which are of order $2^{j} \sim
r^{n_{j}}$. By construction of the sets $\C_j$, the number of random wavelet coefficients at scale $j$
is approximately given by $2^{\gamma j}$. Consequently, there are on average approximately $2^{\eta j}$ non-zero coefficients at scale $j$. The multifractal study of this LWS on $\C(r)$ is provided in the next proposition, proved in Appendix \ref{sec:LWS}. It demonstrates that the multifractal spectrum of $F_{\alpha,\eta,r}$ follows a straight line with the same slope as the classical LWS series $F_{\alpha,\eta}$ on $[0,1]$, with the only difference being in the support of the multifractal spectrum.

\begin{proposition}\label{thm:LWS}
 Let $F_{\alpha, \eta,r}$ be a LWS of parameters $\alpha>0$ and $\eta \in(0,1)$ on the Cantor set $\C(r)$.   Almost surely, one has
  $$
  \mathcal{D}_{F_{\alpha,\eta,r}}(h)  = \rho_{F_{\alpha,\eta,r}}(h)=\begin{cases}
    \frac{ \eta}{\alpha}h & \text{ if } h \in [\alpha, \frac{\alpha\gamma}{\eta}],
    \\[1.5ex]
    1 & \text{ if } h=+\infty, \\[1ex]
  - \infty & \text{ otherwise. }\end{cases}
  $$
\end{proposition}

Note that, in particular, the maximal finite regularity is obtained on a set with the same Hausdorff dimension as that of the underlying Cantor set.
As mentioned earlier, the proofs are rather classical; nonetheless, they are presented in  Appendix  \ref{sec:LWS} for the completeness of the paper and as an introduction for the more technical model studied in Section \ref{sec:LWSduplicated}. Indeed, determining the multifractal spectrum in this case proves to be very useful for obtaining the lower bound of the multifractal spectrum of the ``duplicated'' lacunary wavelet series that we  now define. \\

We  work on the Cantor set $\C(\frac{1}{4})$, whose Hausdorff dimension is $\frac{1}{2}$. Note that the model could easily be adapted to general (symmetric) Cantor sets $\C(r)$ and to other desynchronizations between the wavelets and the steps of construction of Cantor sets. However, the interesting phenomena that we aim to exhibit already appear in this specific case.

\begin{definition}\label{def:dLWS}
    The duplicated LWS on the symmetric Cantor set $\C(\frac{1}{4})$ is defined for $\alpha>0$ and $\eta\in (0,\frac{3}{4})$ as the random wavelet series
    $$
    F^d_{\alpha,\eta}= \sum_{j \in \N} \sum_{k=0}^{2^j-1} c_{j,k} \psi_{j,k}
    \quad \text{with} \quad 
c_{j,k}= \begin{cases}
2^{-\alpha j} \xi_{j,k} & \text{ if }  k \in \Gamma_{j}, \\
0 & \text{ otherwise,}
\end{cases}
$$
where  
$$
\Gamma_j = \left\{ k \in \{0, \dots, 2^j-1\} : \lambda_{j,k} \subset  \C_{\lfloor \frac{j}{4}\rfloor}\right\}
$$
and 
where  $ (\xi_{j,k})_{j,k}$ denotes a sequence of independent random Bernoulli
variables of parameter $2^{(\eta-\frac{3}{4})j}$. 
\end{definition}

Notice that at step $n$, the set $\C_n$ is formed by $2^n$ intervals of length $2^{-2n}$. Each of these intervals contains $2^{\frac{j}{2}}$ dyadic intervals of scale $j=4n$, while they contain only one dyadic interval of scale $j=2n$. The latter case corresponds to the construction of the lacunary wavelet series $F_{\alpha,\eta,\frac{1}{4}}$: with the natural definition of $F_{\alpha,\eta,\frac{1}{4}}$ described above, one would have considered the set  $\Gamma_{j} =\{ k \in \{0, \dots, 2^j-1\} : \lambda_{j,k} \subset  C_{\lfloor \frac{j}{2}\rfloor} \}$.
By not choosing wavelets adapted to the size of the intervals appearing during the construction of the Cantor set, we introduce a duplication of the possible supports for the random coefficients of $F^d_{\alpha,\eta}$. As we will see, this desynchronization creates a richer multifractal behavior with a ``phase transition'' -- the multifractal spectrum is not a straight line -- and it results  in  an overestimation of the multifractal spectrum of $F^d_{\alpha,\eta}$ by its leader large deviation spectrum $\rho_{F^d_{\alpha,\eta}}$.

The main result of the paper is given by the computation of the
exact multifractal spectrum of the lacunary wavelet series $F^d_{\alpha,\eta}$ together with its large deviation spectrum. Its proof is given in Section \ref{sec:LWSduplicated}.

\begin{theorem}\label{thmDLWS}
 Let $F^d_{\alpha, \eta}$ be a duplicated LWS of parameters $\alpha>0$ and $\eta \in(0,\frac{3}{4})$ on $\C(\frac{1}{4})$.
\begin{enumerate}
\item If $\eta \in [\frac{1}{4},\frac{3}{4})$, then almost surely
$$
{\mathcal{D}}_{F^d_{\alpha, \eta}}(h) = \begin{cases}
 \frac{\eta+\frac{1}{4}}{\alpha}h -\frac{1}{2} & \text{ if } h  \in [\alpha, \frac{\alpha}{\eta+\frac{1}{4}}],\\
1 & \text{ if } h= + \infty,\\[1ex]
- \infty & \text{ otherwise}
\end{cases}
$$
and $$
\rho_{F^d_{\alpha, \eta}}(h) = \begin{cases}
\frac{ \eta}{\alpha}h& \text{ if } h  \in [\alpha, \frac{\alpha}{\eta +\frac{1}{4}}],\\[1.5ex]
1 & \text{ if } h= + \infty,\\[1ex]
- \infty & \text{ otherwise.}
\end{cases}
$$
\item If $\eta \in (0, \frac{1}{4}]$, then almost surely
$$
\mathcal{D}_{F^d_{\alpha, \eta}}(h) = \begin{cases}

 \frac{\eta+\frac{1}{4}}{\alpha}h -\frac{1}{2} & \text{ if } h  \in [ \frac{2\alpha}{4 \eta + 1}, 2 \alpha],\\[1.5ex]
\frac{\eta}{\alpha}h & \text{ if }  h \in [2 \alpha, \frac{\alpha}{2 \eta}],\\[1.5ex]
1 & \text{ if } h= + \infty,\\[1ex]
- \infty & \text{ otherwise}
\end{cases}
$$
and $$
\rho_{F^d_{\alpha, \eta}}(h) = \begin{cases}
\frac{\eta}{\alpha}h & \text{ if } h  \in [\alpha, \frac{\alpha}{2\eta}],\\[1.5ex]
1 & \text{ if } h= + \infty,\\[1ex]
- \infty & \text{ otherwise.}
\end{cases}
$$
\end{enumerate}
\end{theorem}

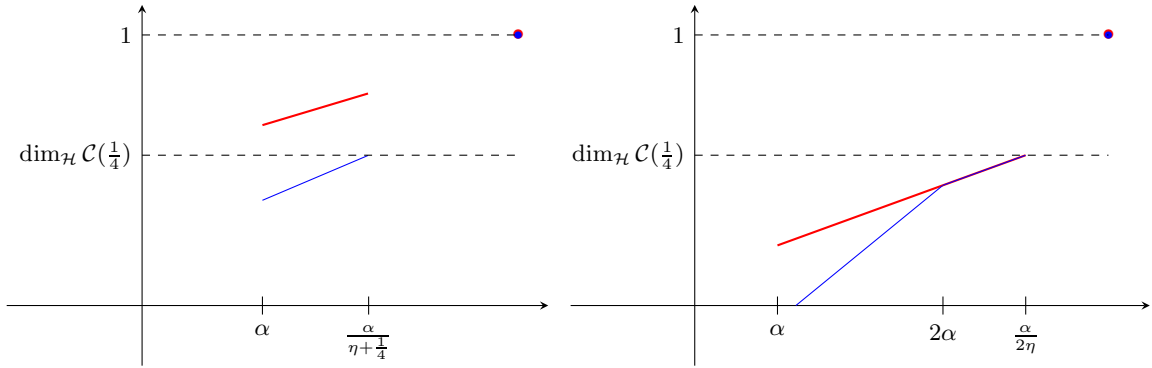
\begin{figure}[ht]
\begin{tikzpicture}
\begin{axis}[
x=20cm,y=4cm,
axis lines=middle,
xmin=-0.09,
xmax=0.27,
ymin=-0.2,
ymax=1,
xtick=\empty,
ytick = \empty,
]
\draw[color=blue, samples=100,domain=0.4:8/17] 
 {plot ({\x -0.32},{(17/8)*(\x) - 0.5})};
 \draw[thick,color=red, samples=100,domain=0.4:8/17] 
 {plot ({\x -0.32},{3*\x/2 })};
\draw[dashed] (0,0.5) -- (0.25,1/2);
\draw  (0,0.5) [left]node {\footnotesize$\dim_{\cal H} \mathcal{C}(\frac{1}{4})$};
\draw  (0.4-0.32,-0.03) -- (0.4-0.32,0.03);
\draw  (0.4-0.32,-0.03) [below]node {\footnotesize$\alpha$};
\draw  (8/17-0.32,-0.03) -- (8/17-0.32,0.03);
\draw  (8/17-0.32,-0.03) [below]node {\footnotesize$\frac{\alpha}{\eta+ \frac{1}{4}}$};
\draw  (0.25,0.9) node[red] {\footnotesize$\bullet$};
\draw  (0.25,0.9) node[blue] {\tiny$\bullet$};
\draw[dashed] (0,0.9) -- (0.25,0.9);
\draw  (0,0.9) [left]node {\footnotesize$1$};
\end{axis}
\hspace{7.5cm}
\begin{axis}[
x=5.5cm,y=4cm,
axis lines=middle,
xmin=-0.3,
xmax=1.1,
ymin=-0.2,
ymax=1,
xtick=\empty,
ytick = \empty,
]
\draw[thick,color=red, samples=100,domain=0.4:1] 
 {plot ({\x -0.2},{\x/2})};
 \draw[color=blue, samples=100,domain=0.8/1.8:0.8] 
 {plot ({\x -0.2},{0.45*\x/0.4 -0.5 })};
  \draw[color=blue, samples=100,domain=0.8:1] 
 {plot ({\x-0.2 },{\x/2})};
 \draw[dashed] (0,0.5) -- (1,1/2);
\draw  (0,0.5) [left]node {\footnotesize$\dim_{\cal H} \mathcal{C}(\frac{1}{4})$};
\draw  (0.4-0.2,-0.03) -- (0.2,0.03);
\draw  (0.2,-0.03) [below]node {\footnotesize$\alpha$};
\draw  (0.6,-0.03) -- (0.6,0.03);
\draw  (0.6,-0.03) [below]node {\footnotesize$2\alpha$};
\draw  (0.8,-0.03) -- (0.8,0.03);
\draw  (0.8,-0.03) [below]node {\footnotesize$\frac{\alpha}{2\eta}$};
\draw  (1,0.9) node[red] {\footnotesize$\bullet$};
\draw  (1,0.9) node[blue] {\tiny$\bullet$};
\draw[dashed] (0,0.9) -- (1,0.9);
\draw  (0,0.9) [left]node {\footnotesize$1$};
\end{axis}
\end{tikzpicture}
\caption{The leader large deviation spectrum (in red)  and  the multifractal spectrum (in blue) of $F^d_{\alpha, \eta}$, with $\alpha
        =0.4$ and $\eta =0.6$ (left), $\eta =0.2$  (right). }
\end{figure}

\begin{remark}\
    \begin{itemize}
  \item In both cases, the Hausdorff dimension of the iso-H\"older set  of $F^d_{\alpha,\eta}$ with maximal finite regularity matches exactly the Hausdorff dimension of the Cantor set $\C(\frac{1}{4})$. 
  \item One could of course replace in the model the coefficients $c_{j,k}=0$
  by $c_{j,k} = 2^{-\gamma j}$ with any exponent $\gamma$ larger than the
  maximal regularity of $F^d_{\alpha,\eta}$ on the Cantor set $\C(\frac{1}{4})$.  This  would yield
  $\mathcal{D}_{f}(\gamma) = \rho_{f}(\gamma) = 1$.  
  \end{itemize}
\end{remark}

Let us examine the role of the lacunarization parameter $\eta$. It can be observed that for each $\eta$, the maximum of the multifractal spectrum $\mathcal{D}_{F^d_{\alpha, \eta}}$ on $[0, + \infty[$ always equals the dimension of the Cantor set $\C(\frac{1}{4})$. Conversely, for $\eta>\frac{1}{4}$,  the maximum of the leader deviation spectrum $\rho_{F^d_{\alpha, \eta}}$ is given by $\frac{\eta}{\eta+ \frac{1}{4}}$, which decreases as the series becomes more lacunar, i.e. as $\eta$ decreases, until the value $\eta = \frac{1}{4}$ is reached. At this point, it reaches the dimension of the Cantor set and does not decrease anymore.

Essentially, it shows that lacunarization enhances the effectiveness of wavelet leaders, in the sense that it allows the Cantor structure to emerge more distinctly. This observation serves as the starting point for the second part of this paper, which represents the initial step towards addressing the general question: given a function $f$ -- signal or image -- analyzed in a wavelet basis, is there a procedure involving exclusively computable quantities such as the leader large deviation spectra on ``lacunarized expansions'' of $f$ that could indicate an overestimation of ${\mathcal{D}}_f$ by $\rho_f$ ? The second part of the paper shows that for \emph{$\alpha$-sparse wavelet series}, as defined just below, computing the maximum of the leader large deviation spectrum of the image of the considered function under a lacunarization operator provides a criterion to detect such an overestimation. To introduce this, we fix a wavelet basis $\psi_{j,k}$ that is sufficiently  regular.

\begin{definition}\
\begin{enumerate}
\item A wavelet series $\sum_{j \in \N}\sum_{k=0}^{2^j-1} c_{j,k}\psi_{j,k}$ is said to be $\alpha$-sparse if $c_{j,k}$ either equals to $2^{-\alpha j}$ or $0$. 
\item For $\eta>0$ and a uniformly H\"older function $f = \sum_{j \in \N}\sum_{k=0}^{2^j-1} c_{j,k}\psi_{j,k}$, we define the lacunarization opertor $L_{\eta}$ by setting
$$
 L_{\eta}(f) = \sum_{j,k} c_{j,k} \xi^{\eta}_{j,k}\psi_{j,k}
$$
where $(\xi^{\eta}_{j,k})_{j,k}$ denotes a sequence of independent random Bernoulli variables of parameter $2^{(\eta-1) j}$.
\end{enumerate}
\end{definition}
Note that the closer $\eta$ is to $0$, the more the function $f$ is lacunarized through the operator $L_\eta$, meaning that the number of non-empty coefficients at each scale get smaller. \\

Particular $\alpha$-sparse wavelet series we will first study can be constructed by considering a decreasing sequence of closed sets $(\I_j)_{j \in \N}$ of $ [0,1]$ satisfying 
$$
\text{dim}_{\mathcal{H}} \Big( \bigcap_{j \in \N}\I_j \Big) > 0. 
$$
For every $\alpha>0$, we naturally associate to the sequence $(\I_j)_{j \in \N}$ a $\alpha$-sparse wavelet series $F^{(\alpha,(\I_j))}$ by specifying its expansion in a fixed wavelet basis in the following way 
$$
F^{(\alpha,(\I_j))} = \sum_{j \in \N}\sum_{k=0}^{2^j-1} c_{j,k} \psi_{j,k}  \quad \text{with}\quad 
c_{j,k}= \begin{cases} 2^{-\alpha j} & \text{ if } \lambda_{j,k} \, \cap \, \I_j \neq \emptyset \\ 0 & \text{ otherwise.} \end{cases}
$$

\begin{remark}
  Lacunary wavelet series on $[0,1]$, or on a Cantor set $\C(r)$, and duplicated lacunary wavelet series can all be expressed as a lacunarization of such a function, where, for all $j \in \N$, 
  \begin{enumerate}
  \item   $\I_j =[0,1]$ for the standard LWS,
   \item $\I_j = \bigcup \big\{ \lambda_{j,k} \subset  \C_{n_j}\big\}$ with $n_j = \lfloor \frac{-j}{\log_2 r}\rfloor$ for the LWS on Cantor set,
  \item $\I_j = \bigcup \big\{\lambda_{j,k} \subset \C_{\lfloor \frac{j}{4} \rfloor}\big\}$ for the duplicated lacunary wavelet series on $\C(\frac{1}{4})$.
  \end{enumerate}
\end{remark}

Note that the nest of sets $\I_j$ implies the strict hierarchy of the wavelets coefficients, that is $c_{\lambda'} \le c_{\lambda}$ as soon as $\lambda'  \subset \lambda$. Furthermore, we set
$$
\G_j = \big\{ \lambda_{j,k} \, :  \lambda_{j,k}\,  \cap\,  \I_j \neq \emptyset  \big\}.
$$
The second major result of the paper is the following.

\begin{theorem}\label{TheoDetect} Fix $\alpha >0$, and let $(\I_j)_{j \in \N}$ be a decreasing sequence of  subsets of $[0,1]$. Define $\I = \bigcap_{j \in \N} \I_j$, $H = \dim_{\mathcal{H}}(\I)$, and  assume that $H >0$. If  $F^{(\alpha,(\I_j))}$ satisfies the leader large deviation formalism, i.e. ${\mathcal{D}}_{F^{(\alpha,(\I_j))}} = \rho_{F^{(\alpha,(\I_j))}}$, then 
\begin{enumerate}
\item the sequence $(\I_j)_{j \in \N}$ satisfy a statistical self-similarity at large scale, that is : for any $\beta, \varepsilon>0$, there exists $J \in \N$ such that for all $j \ge J$, one has 
\begin{equation*}
    2^{j(H-2\varepsilon)} \le \# ND(j,\beta, \varepsilon) \le 2^{j(H +\varepsilon)}
\end{equation*}
where the set $ND(j,\beta, \varepsilon) $ is defined as
$$
ND(j,\beta, \varepsilon) := \big\{ \lambda\in \G_j: \,  2^{j(\beta H -4\max(1,\beta)\varepsilon)} \le  \# \{ \lambda' \subset \lambda : \lambda' \in \G_{\lfloor (1+\beta)j \rfloor} \} \le  2^{j(\beta H +4\max(1,\beta)\varepsilon)} \big\};
$$
\item for every $\eta \in (1-\dim_{\cal{H}}(\I),1]$, one has
$$
\max_{0<\beta<+\infty} \rho_{L_\eta(F^{(\alpha,(I_j))})}(\beta) = \dim_{\cal H}(\I)\, ;
$$
\item for every $\eta \in (0, 1- \dim_{\cal{H}}(\I)]$, one has 
$$
\max_{0<\beta<+\infty} \rho_{L_\eta(F^{(\alpha,(I_j))})}(\beta) \in \{ - \infty, 0\}. 
$$
\end{enumerate}
\end{theorem}

\begin{remark}\
\begin{itemize}
\item Roughly speaking, the sets $ND(j,\beta, \varepsilon)$ consist of the collection of dyadic intervals $\lambda \in \I_j$ that have been duplicated at the scale $\lfloor (1+\beta)j \rfloor$ with an expected rate determined by the Hausdorff dimension of $\I$. The first part of the theorem states that the cardinality of this set provides an estimation of $\dim_{\mathcal{H}}(\I)$ up to  $\varepsilon$. This  is what we mean by ``statistical similarity''.
\item The second part of the theorem states that for such hierarchical signals, a decrease in the maximum value of $\rho$ through the lacunarization process implies that the leader large deviation formalism is not valid.
\end{itemize}
\end{remark}
We derive two corollaries from Theorem \ref{TheoDetect}. Let us recall that the upper box-counting dimension of a set $\I$ is defined by
$$
\dim_{u-box}(\I) := \limsup_{\varepsilon \to 0^+} \frac{\log N(\varepsilon)}{-\log(\varepsilon)}
$$
where $N(\varepsilon)$ is the smallest number of intervals of length $\varepsilon$ which can cover $\I$ (see \cite{Falconer:86} for an introduction to box-counting dimensions). The first result is obtained immediately by contraposing Part 2 of Theorem \ref{TheoDetect}. One of its main interest is that it does not depend anymore on the choice of the wavelet basis. It can particularly be of interest when performing a multifractal analysis of data given on a fractal set, as is the case for urban data for example (see \cite{Len:22}).
 
\begin{corollary}
Let $\I \subset [0,1]$ be a set with strictly positive Hausdorff dimension. Let $(\I_j)_{j \in \N}$ be the decreasing sequence of subsets of $[0,1]$ defined by
$$
\I_j = \Big\{  \left[ 2^{-j}k, 2^{-j}(k+1) \right]\, : \,  \left[ 2^{-j}k, 2^{-j}(k+1) \right] \cap \I \neq \emptyset \Big\}.
$$
If the strictly positive maximal values of $\rho_{L_\eta F^{(\alpha,(\I_j))}}$ on $\R$ are not the same for all $\eta$, then upper-box counting dimension and the Hausdorff dimension of $\I$ do not coincide, i.e.
$$
\dim_H(\I)< \dim_{u-box}(\I).
$$
\end{corollary}

The second consequence of Theorem \ref{TheoDetect} is a procedure to detect an overestimation of the singularity spectrum by the leader large deviation spectra for $\alpha$-sparse wavelet series. It is only based on computations of leader large deviation spectra of a family of functions derived from the original analyzed one. 
Let $\alpha>0$ be fixed. If $f$ is $\alpha$-sparse wavelet series, we ``saturate'' the wavelet series by setting
\begin{equation}\label{eq:sature}
    S(f) = \sum_{j\in \N} \sum_{k=0}^{2^j - 1 } s_{j,k} \psi_{j,k} \, \text{ where } \,  s_{j,k} = \begin{cases}
    2^{-\alpha j} & \text{ if } \sup_{\lambda' \subset \lambda_{j,k}} |c_{\lambda' }| \neq 0\\[1ex]
    0 &\text{ otherwise. } 
\end{cases}
\end{equation}

\begin{corollary}\label{cor:validity}
Let $f$ be a $\alpha$-sparse wavelet series such that 
$$
{\rm Supp } \, \mathcal{D}_f  \subset [h_{min},h_{max}] \cup \{+\infty\}.
$$
for some $0 <h_{min} \le h_{max}<+\infty$.
If $\mathcal{D}_f = \rho_f$ and if $\alpha_0 \in [h_{min},h_{max}] $ is such that
$$
\max_{h_{min}\le \alpha \le h_{max}} \rho_f(\alpha)= \rho_f(\alpha_0),
$$ 
then 
$$
\max_{0<\alpha <+\infty} \rho_{L_{\eta}S(f)}(\alpha) = \rho_f(\alpha_0)
$$
for all $\eta \in (1- \rho_f(\alpha_0),1] $.
\end{corollary}

Consequently, if there exists a lacunarization parameter $\eta$ for which the maximum of the leader large deviation spectrum of function $L_{\eta}S(f)$ is strictly smaller than $\rho_f(\alpha_0)$, it implies that  $\mathcal{D}_f \neq \rho_f$. This corollary also paves the way to a new multifractal formalism, i.e. a new formula to estimate  multifractal spectra. 
This procedure will be further investigated in an upcoming paper.

\section{Wavelet leader based mutlifractal formalisms}\label{sec:formalisms}

\subsection{Legendre spectrum and Large deviation spectrum}

The multifractal spectrum remains an abstract quantity when dealing with digital data, making its computation seemingly impossible {\it a priori}. In \cite{Parisi:85}, Frish and Parisi suggest estimating this spectrum using computable quantities like increments, which offer statistical insights into changes in pointwise regularity. Building upon this concept, numerous numerical estimation techniques, referred to as \emph{multifractal formalisms}, have emerged, as seen in \cite{Jaffard:00,Jaffard:04,Aubry:07,Jaffard:97,Muzy:93,Abry:14,Kleyntssens:18}. These methods all give an upper bound of the multifractal spectrum. In this paper, we will focus on two recent and widely used formalisms, both based on a characterization of the  pointwise H\"older regularity using \emph{wavelet leaders}, see \cite{Jaffard:04b,Jaffard:05,JLA}.

\begin{definition}\label{def:leader} Let $\lambda$ be a dyadic interval and $3 \lambda$ the interval of same center  as $\lambda $ and 3 times wider. If $f$ is a bounded function, the wavelet leader $d_{\lambda}$ of $f$ is defined by
$$
d_{\lambda} = \sup_{\lambda' \subset 3 \lambda} \vert c_{\lambda'} \vert.
$$
\end{definition}

The significance of these quantities lies in their ability to determine the pointwise H\"older regularity of a function $f$ at a point $x_0$, as detailed in \cite{Jaffard:04b}. For $x_0 \in \mathbb{R}$, the notation $\lambda_j(x_0)$ denotes the unique dyadic interval of width $2^{-j}$ containing $x_0$.
 
\begin{theorem}\cite{Jaffard:04b}\label{thm:waveletcharact}
Let $h>0$ and $x_0 \in \R$. Assume that $f$ is a bounded function and that the wavelet has $r$ vanishing moments  with $r > \lfloor h \rfloor +1$. 
\begin{itemize}
\item If $f$ belongs to $\C^{h}(x_0)$, then there exists $C>0$ such that
\begin{equation}\label{leaderholder}
\forall j \ge 0, \quad d_{\lambda_j(x_0)} \le C2^{-hj}.
\end{equation}
\item Conversely, if (\ref{leaderholder}) holds and if $f$ is uniformly H\"older  ({\it i.e.} $f$ belongs to $\C^{\varepsilon}(\R)$ for some $\varepsilon >0$), then $f$ belongs to $C^{h'}(x_0)$ for all $h'<h$. 
\end{itemize}
In particular,  if $f \in \C^{\varepsilon}(\R)$ for some $\varepsilon>0$,  then $$h_f(x_0) = \liminf_{j\to +\infty} \frac{\log d_{\lambda_j(x_0)}}{\log 2^{-j}}.$$
\end{theorem}

The latter characterization motivates the introduction of the following formalism. First, one defines the structure function 
$$
S^{f}_j(p) = 2^{-j} \!\sum_{\lambda \in \Lambda_j, d_\lambda \neq 0}\! (d_{\lambda})^p
$$
for $j \in \N$ et $p \in \R$, 
where $d_{\lambda}$ denote the wavelet leaders of the function $f$
under study. 
The scaling function is then given by
$$
\eta_{f}(p)= \liminf_{j \to +\infty}  \frac{\log S^{f}_j(p)}{\log 2^{-j}} 
$$
and finally, the so-called \emph{leader Legendre spectrum} is defined by 
$$
L_f(h) = \inf_{p \in \R} (1-\eta_{f}(p)+hp).
$$
The formalism consists then in the estimation of the multifractal spectrum via the leader Legendre spectrum: if the equality $\mathcal{D}_f= L_f$ is true, one says that the leader Legendre formalism holds for $f$. 
The properties of the leader Legendre spectrum are recalled in the following proposition (see Proposition 5 of \cite{Jaffard:10a} for example).

\begin{proposition}\label{prop:admissible}
The leader Legendre spectrum $L_f$ is a concave function. If we suppose that there exists $C_1,C_2,A,B$ such that
\begin{equation}\label{encadrement}
\forall j \in \N, \, \forall \lambda \in \Lambda_j, \quad C_1 2^{-Bj} \le d_{\lambda} \le C_2 2^{-Aj}
\end{equation}
and if we denote by $H_{\max} :=  \inf\{ A>0 : \, (\ref{encadrement}) \text{ holds for  some }C_{2}\}$ and $H_{min} := \sup
\{B>0 : \,  (\ref{encadrement}) \text{ holds for some } C_{1} \}, $
 then $L_{f}$
satisfies
\begin{itemize}
\item $0 \le L_f\le 1$ on  $[H_{\min}, H_{\max}]$ and  $L_{f} = -
  \infty$ otherwise, 
\item there exists $H_1$ and $H_2$ such that $L_{f}$ is strictly increasing on $[H_{\min}, H_{1}]$, strictly decreasing on $[H_2, H_{\max}]$ and constant equal to $1$ on $[H_1, H_2]$.
\end{itemize}

\end{proposition}

Since $L_f$ is a concave function,  the leader Legendre formalism can only hold for concave multifractal spectrum. To address this problem, another formalism based on large deviation estimates of wavelet leaders has been derived, see \cite{Bastin:16,Esser:17}.

\begin{definition}\label{def:largedeviation} 
The \emph{leader large deviation spectrum}  of $f$ is defined for every $h \ge 0$ by
\[
\rho_f(h)= \lim_{\varepsilon \to 0^+} \limsup_{j \to +\infty}
\frac{ \log \# \{ \lambda \in \Lambda_j : \, 2^{-(h+ \varepsilon)j}
  \leq d_{\lambda} < 2^{-(h-\varepsilon)j}\} }{\log 2^j}
\]
and for $h=+\infty$ by
\[
\rho_f(+\infty) = \lim_{A \to + \infty} \limsup_{j \to +\infty}
\frac{\log \# \{ \lambda \in \Lambda_j : \, d_{\lambda}< 2^{-Aj}\}}{\log 2^j}.
\]
\end{definition}

We say that the leader large deviation formalism is satisfied for a function $f$ if $f$ satisfies the equality $\mathcal{D}_f = \rho_f$. Note that the leader large deviation spectrum is an upper semi-continuous function on $[0, +\infty)$, and its maximum is equal to $1$. Furthermore, it can be equivalently defined using the so-called \emph{restricted wavelet leaders $e_\lambda$} of $f$, defined by
$$
e_{\lambda} = \sup_{\lambda' \subset  \lambda} \vert c_{\lambda'} \vert.
$$
If $f$ is a uniformly H\"older function, a connection between the leader large deviation spectrum $\rho_f$, the leader Legendre spectrum $L_f$, and the multifractal spectrum $\mathcal{D}_f$ can be established.

\begin{proposition}\label{prop:comparison}\cite{Esser:17}
If $f$ is uniformly H\"older and if $\rho_f = - \infty$ 
outside a compact set, then
\[
\mathcal{D}_{f}\le \rho_f \le L_f.
\]
In addition,  $\rho_f = L_f$ if and only if $\rho_f$ is concave.
\end{proposition}

This inequality shows that in general, the leader large deviation spectrum is more accurate than the leader Legendre spectrum for estimating the multifractal spectrum. 
Finally, let us define the increasing wavelet leader profile as done in \cite{Bastin:16}.

\begin{definition}\label{def:rhocum}
    The increasing wavelet leader profile of a bounded function $f$ is  defined for every $h \in [ 0,+\infty )$  by 
    $$
    \rho_{{\rm cum},f}(h)= \lim_{\varepsilon \to 0^+} \limsup_{j \to+\infty} \frac{\log \# \big\{ \lambda \in \Lambda_j :  d_{\lambda} \ge 2^{-j(h+\varepsilon)}\big\}}{\log 2 ^j} .
    $$
  \end{definition}
The increasing wavelet leaders profile has the advantage to be robust to the change of the wavelet basis, which is not the case for the  leader large deviation spectrum. It is an increasing an right-continuous function on $[ 0,+\infty ) $ that satisfies 
\begin{equation}\label{eq:rhocum}
    \rho_{{\rm cum},f}(h) = \sup_{h' \le h} \rho_f(h') \quad \forall h \in [0,+\infty),
\end{equation}
see \cite{Bastin:16}. Moreover, as for the leader large deviation spectrum, the increasing leader profile of $f$ can be equivalently be defined using its restricted wavelet leaders.

\subsection{Validity of the formalisms}

In general, a formalism is a method based on numerically computable quantities that provides an upper bound for the multifractal spectrum of a given function. If the inequality turns out to be an equality for a specific function $f$, one says that the formalism is satisfied for $f$. 
In particular, as explained in the previous subsection, a  function $f$ satisfies the leader Legendre formalism if ${\mathcal{D}}_f = L_f$, and it satisfies the leader large deviation formalism if ${\mathcal{D}}_f = \rho_f$. Clearly, functions or processes that do not satisfy to the formalism based on the leader large deviation spectrum also do not satisfy to the formalism based on the leader Legendre spectrum.

Many counterexamples to the leader Legendre and/or the leader large deviation formalisms have been  identified, in particular for measures (that can be identified with wavelet series), see e.g. \cite{Nasr:02,Testud:06,Shmerkin05,Feng05,BaSe23}. Recently, Barral and Seuret have studied the generic validity of the leader Legendre formalism, and in particular have constructed function spaces in which the latter formalism is generically not valid, see \cite{BaSeFPI,BaSeFPII}.\\

The main result of this paper, Theorem \ref{thmDLWS}, presents a new counterexample based on the principle of ``duplication'' of the wavelet coefficients. Before studying in details this result, we illustrate in this brief subsection how duplicating wavelet coefficients can influence the multifractal spectrum of a function through a toy-example. The construction strategy begins with any function $f$ that satisfies the formalism, and involves creating a new wavelet series by  sticking together several copies of the wavelet leaders of $f$. While the H\"older regularity of this new function remains controlled by the reguarity of $f$, its leader large deviation spectrum can be altered arbitrarily close to $1$.

Let $f$ be any uniformly H\"older function such that 
$
\rho_{f} = \mathcal{D}_{f}.
$
We denote by $c_{\lambda}$ its wavelet coefficients and by
$d_{\lambda}$ its wavelet leaders.  For every real $m>1$, we consider the wavelet
series $f_{m}$ defined by
\begin{equation}\label{eq:gm}
f_{m} = \sum_{j \in \N}\sum_{k=0}^{2^{j}-1} C^{m}_{j,k} \psi_{j,k} \quad \text{with} \quad 
C^{m}_{j,k} = \sup_{\lambda' \subset \lambda^{\lfloor
\frac{j}{m} \rfloor}_{j,k}}|c_{\lambda'}|,
\end{equation}
where $\lambda^{\lfloor \frac{j}{m} \rfloor}_{j,k}$ denotes the unique dyadic interval of scale $\lfloor
\frac{j}{m} \rfloor$ that contains $\lambda_{j,k}$. Clearly, $f_{m}$ still
satisfies a uniform H\"older condition and the wavelet series defining
it is convergent. Additionally, the sequence of wavelet coefficients
of $f_{m}$ is hierarchical, implying that its wavelet leaders, denoted here as $D^{m}_{j,k}$, are simply
given by $D^{m}_{j,k} =d_{\lambda^{\lfloor \frac{j}{m} \rfloor}_{j,k}}$.

\begin{proposition}\label{prop:gm}
Let $f$ be any uniformy H\"older function such that 
$
\rho_{f} = \mathcal{D}_{f},
$
fix $m>1$ and let  $f_{m}$ be the function defined by \eqref{eq:gm}.
For every $h \ge 0$, one has
$$\rho_{f_{m}}(h) = \frac{m -1 +\mathcal{D}_{f_{m}}(h) }{m}.$$
\end{proposition}

\begin{proof}
If a  wavelet leader of $f$ of scale $\lfloor \frac{j}{m}
  \rfloor$ of $f$ is of order $2^{-h j}$, it generates $2^{j- \lfloor \frac{j}{m}
  \rfloor}$ wavelet leaders of $f_m$ of scale $j$ with the same order. Hence
one has
$$
\# \big\{\lambda \in \Lambda_{j} : 2^{-(h + \varepsilon)j} \leq
  D^{m}_{\lambda} < 2^{-(h - \varepsilon)j}\big\}\\
 = 2^{j- \lfloor \frac{j}{m}   \rfloor} \# \big\{\tilde{\lambda} \in
      \Lambda_{\lfloor \frac{j}{m}   \rfloor} : 2^{-\frac{(h m + \varepsilon m )j}{m}} \leq
  d_{\tilde{\lambda}} < 2^{-\frac{(h m - \varepsilon m )j}{m}}\big\}
$$
and consequently,
$$\rho_{f_{m}}(h) 
 =  1- \frac{1}{m}+ \frac{\rho_{f}(hm)}{m} 
 =  \frac{m-1 +\mathcal{D}_{f}(hm) }{m}
$$
since by assumption, $\rho_{f}=\mathcal{D}_{f}$. 
To conclude, we need to show that $\mathcal{D}_{f_{m}}(h) =
\mathcal{D}_{f}(hm)$. For every $x \in [0,1]$, one has $D^{m}_{\lambda_{j}(x)} = d_{\lfloor \frac{j}{m} 
  \rfloor}(x)$, leading to
$$
\liminf_{j \to + \infty} \frac{\log  D^{m}_{\lambda_{j}(x)} }{\log 2^{-j}} =\frac{1}{m} \liminf_{j \to + \infty} \frac{\log  d_{\lfloor \frac{j}{m} 
  \rfloor}(x)}{\log 2^{-\lfloor \frac{j}{m}   \rfloor}}.
$$
Using  the wavelet characterization of the
H\"older exponent given in Theorem  \ref{thm:waveletcharact}, we
obtain 
 $ h_{f_{m}}(x) =
\frac{h_{f}(x)}{m}$, hence the conclusion. 
\end{proof}

This approach offers numerous counterexamples to formalisms, including those involving so-called \emph{H\"older-homogeneous functions}. Let us briefly recall this notion, see \cite{BDJS13, Kae:13} for more precisions.  Given a nonempty open subset $\Omega \subset \mathbb{R}$, the \emph{$\Omega$-local multifractal spectrum} of a function $f$ is defined as $$\mathcal{D}_{f}^{\Omega}(h) = \dim_{\mathcal{H}}(I_f(h) \cap \Omega).$$ A function $f$ is called \emph{H\"older-homogeneous} if its $\Omega$-local multifractal spectrum is independent of $\Omega$. The existence of H\"older-homogeneous counterexamples to multifractal formalisms was an open question raised in \cite{Jaffard:10a}. Indeed, the homogeneity of a function is often regarded as a means to ensure the validity of the leader Legendre  formalism \cite{Seuret:13}. One says that a function $L$ is \emph{an admissible leader Legendre spectrum} if it  satisfies the conditions of Proposition
\ref{prop:admissible}. Given such an admissible leader Legendre spectrum $L$, constructions of functions $f$ such that $L_f=L$ and which satisfy the leader Legendre formalism (and hence the leader large deviation formalism) have been proposed by Jaffard in \cite{Jaffard:92}, and more recently by Coiffard, Melot, and Willer in \cite{Coiffard:14}. It is easy to verify that these constructions are H\"older-homogeneous. Starting with such a function $f$, the procedure described in this subsection yields a family of functions $f_m$ defined by \eqref{eq:gm}, which are still H\"older-homogeneous but possess a leader Legendre spectrum that differs from their multifractal spectrum.

\begin{corollary} Let $L$ be an admissible leader Legendre spectrum whose
  support is not reduced to a single point and such that
  $L(H_{min})>0$ and $L(H_{max}) >0$. Then, there exists a H\"older-homogeneous 
  function  $g$ such that   $L_{g} = L$ and 
$$
\mathcal{D}_g \neq L_g .
$$
\end{corollary}
\begin{proof}  
Let us fix $m$ such that $1 < m \le
  \frac{1}{1-\min(L(H_{min}), L(H_{max}))}$. It is direct to show that the function
  $\widetilde{L}=m L +1-m$ is also an admissible leader Legendre spectrum. Consequently, the function   $\widetilde{L} (\frac{\cdot}{m})$ does as well. Thus, there exists a H\"older-homogeneous function $f$ which satisfies $\mathcal{D}_{f} =
  \widetilde{L} (\frac{\cdot}{m})$. Propositions \ref{prop:comparison} and \ref{prop:gm} imply then
  that the associated function $f_m$ constructed
  in \eqref{eq:gm} is a H\"older-homogeneous function with
  $L_{f_m}= L$ and $\mathcal{D}_{f_m} = m L +1-m$.   
\end{proof}

\section{Duplicated LWS : Proof of Theorem \ref{thmDLWS}}\label{sec:LWSduplicated}

This section is dedicated to proving the first main result of the paper, namely Theorem \ref{thmDLWS}. It  will be divided into several results that collectively lead to the proof of this theorem.

Let $F^d_{\alpha,\eta}$ be the duplicated LWS of parameters $\alpha>0$ and $\eta \in (0, \frac{3}{4})$ on the Cantor set $\C(\frac{1}{4})$ as introduced in Definition \ref{def:dLWS}. In order to simplify the notations in this section, we set $\mathcal{K}=\C(\frac{1}{4})$. We will also identify the set of indices
$$
\Gamma_j = \left\{ k \in \{0, \dots, 2^j-1\} : \lambda_{j,k} \subset \C_{\lfloor \frac{j}{4}\rfloor}\right\}
$$
with  the subset of $\Lambda_j$ defined by
$$
\Big\{\lambda \in \Lambda_j : \lambda \subset \C_{\lfloor \frac{j}{4}\rfloor}\Big\}.
$$
Let us recall that  $\C_n$ is formed by $2^n$ intervals of length $2^{-2n}$.  If $n =\lfloor \frac{j}{4}\rfloor$, each of these intervals contains $2^{2n+ (j \, {\rm mod }\, 4)}$ dyadic subintervals of scale $j$. Hence
\begin{equation}\label{eq:cardGamma}
    \# \Gamma_j = 2^n2^{2n+( j \, {\rm  mod } \, 4)} = 2^{3n+ (j \text{ mod } 4)} = 2^{\frac{3}{4}j}2^{\frac{j \, {\rm  mod }\, 4}{4}}. 
\end{equation}

\begin{remark}
At first sight, the model studied here may appear similar to the one defined via Equation~\eqref{eq:gm}, as the chosen scales $j=4n$ are multiples of the natural ones $2n$. However, a notable distinction lies in the duplication of the supports of the random coefficients, rather than the non-zero coefficients themselves. Another difference lies in the fact that we do not work directly on the wavelet leaders, which is more natural, and we introduce randomness. Note that for the specific case of a classical LWS on $\mathcal{K}$, the procedure of  Equation~\eqref{eq:gm} with $m=2$ would result in a multifractal spectrum given by $2\eta \frac{h}{\alpha}  \mathbf{1}_{[\frac{\alpha}{2}, \frac{\alpha}{4\eta}]}(h)$. 
\end{remark}

\subsection{Estimation of the number of non-zero coefficients}

Before delving into the details of the proof of Theorem \ref{thmDLWS}, let us introduce some notations and  provide some comments regarding the number of non-zero or random coefficients of $F^d_{\alpha,\eta}$, all summarized in the following Lemmas.

\begin{lemma}\label{lem:P_j}
For every $j \ge 0$,  let $A_j$ be the random subset of $\Gamma_j$ defined as
$$
A_j = \{ \lambda \in \Gamma_j : \, c_{\lambda} = 2^{-\alpha j} \}
$$
representing the positions of the non-zero coefficients of  $F^d_{\alpha,\eta}$. 
Then, one has
 $$
\EE[\# A_{j}] = 2^{\eta j} 2^{\frac{j\, {\rm  mod }\, 4}{4}}.
$$
\end{lemma}

\begin{proof}
It is direct using \eqref{eq:cardGamma}, since there are $\# \Gamma_j$ random coefficients at scale $j$ that are non-zero with a probability of $2^{(\eta-\frac{3}{4})j}$.
\end{proof}

One could then expect that the  leader large
deviation spectrum  of  $F^d_{\alpha,\eta}$ evaluated at $\alpha$ would be equal to $\eta$. This is precisely the result we will obtain in Theorem  \ref{thmDLWS}.  

\begin{lemma}\label{lem:R_j}
 For every $j \ge 0$,  let $R_j$ be the  subset of $\Gamma_j$  defined as 
$$R_j = \{ \lambda \in \Gamma_j :  \lambda \cap \K \neq \emptyset
\} .$$
 Then, one has
$$
\# R_j = 2^{\frac{j}{2}} 2^{\frac{j \, {\rm  mod }\, 4}{2}}.$$
\end{lemma}

\begin{proof}
Note that if $j=4n$, the dyadic intervals of $\Gamma_j\setminus R_j$ are precisely the intervals that will be removed in the subsequent steps of the construction of the Cantor set $\K$, up to step $2n$, as the set $\C_{2n}$ contains dyadic intervals of length $2^{-j}$. In particular, we have $$R_j = \big\{ \lambda \in \Lambda_{j} : \lambda \cap \C_{2n}\neq \emptyset\big\}$$ and consequently,
$
\# R_j = 2^{\frac{j}{2}} =2^{2n}.$
If $j \in \{4n+1, 4n+2, 4n+3 \}$, then each dyadic interval of scale $4n$ that will be removed give birth to $2^{j \, {\rm  mod }\, 4}$ subintervals of $R_j$. The conclusion follows. 
    \end{proof}

Note that even though their support will be removed in the construction of the Cantor set $\K$, the coefficients associated with intervals of $\Gamma_j \setminus R_{j}$ still influence the wavelet leaders of points in $\K$ and consequently, affect the multifractal spectrum of $F^d_{\alpha,\eta}$. This will become evident in the proof of Proposition \ref{prop_sup} below.

\begin{remark}\label{rem:R_j}
Consider $\lambda \in R_j$ with $j= 4n$. If the wavelet leader $d_\lambda$ of $F^d_{\alpha,\eta}$ is equal to $2^{-hj}$ for some $h >\alpha$, it implies the existence of a non-zero wavelet coefficient $c_{\lambda_{0}} = 2^{-\alpha j_{0}} = 2^{-hj}$ of scale $j_{0}= \frac{h}{\alpha}j$.
\begin{itemize}
    \item If $\lambda_0 \subset \lambda$ is a dyadic interval of scale $j_{0} \leq 2j$, then $\lambda_0 \in \Gamma_{j_0}$.  Note that 
    $$
    \#\{\lambda_0 \subset \lambda : \lambda_{0} \in \Gamma_{j_{0}}\} = 2^{j_{0}-j}.
    $$
    \item If we now consider the scale $j_0=2j+4$, only the dyadic subintervals of $\lambda$ of scale $j_{0}$ included in the first and the last quarter of $\lambda$ remain in $\Gamma_{j_0}$. Consequently,
    $$
    \#\{\lambda_0 \subset \lambda : \lambda_{0} \in \Gamma_{j_{0}}\} = 2^{j+3} < 2^{j_{0}-j}.
    $$
\end{itemize}
This explains the different behaviors in the computation of the multifractal spectrum, depending on the relative position of $h$ with respect to $2\alpha$. Note that the second case corresponding to $h \geq 2\alpha$ will never occur if the series is not too lacunar, that is if $\eta \geq \frac{1}{4}$. In this case, the maximal regularity of $F^d_{\alpha,\eta}$  at the points of the Cantor set will be smaller than $2\alpha$, as indicated in Lemma \ref{lem:maximal_regularity}.
\end{remark}

Lemma \ref{lem:R_j} directly implies the following estimation of the number of non-zero coefficients associated with dyadic intervals of $R_j$.

\begin{lemma}\label{lem:G_j}
 For every $j \ge 0$,  let $B_{j}$ be  the random
subset of $R_{j}$ defined by 
$$
B_j = A_j \cap R_j =  \{ \lambda \in R_j : c_{\lambda} = 2^{-\alpha j}\}.
$$
 Then, one has
$$
\EE[\# B_{j}] = 2^{(\eta-\frac{1}{4}) j}2^{\frac{j \, {\rm  mod }\, 4}{2}}. 
$$
\end{lemma}

Here, we observe that if $\eta < \frac{1}{4}$, there will be very few wavelet coefficients of order $2^{-\alpha j}$ corresponding to dyadic intervals of $R_j$ at every scale $j$. More precisely, the supremum of the number of non-zero coefficients at scale $j$ whose support intersects $\K$ will be almost surely bounded in $j$, as proven in Lemma \ref{lem:cardFj} below. In particular, the regularity $\alpha$ might not be attained. This is further confirmed by Lemma \ref{lem:minimal_regularity}, which demonstrates that in this case, the minimal regularity is $\frac{2\alpha}{4\eta +1}$.

The computation of the expectations of $A_{j}$ and $B_{j}$ provided in Lemmas \ref{lem:P_j} and \ref{lem:G_j}, together with the Chebyshev inequality combined with the Borel-Cantelli lemma, directly yields the following lemma.

\begin{lemma}\label{lem:cardFj}
Almost surely, for every $\varepsilon >0$, there is $J \in \N$ such
that
$$
2^{( \eta - \varepsilon) j} \le \# A_{j} \leq 2^{( \eta + \varepsilon) j} \quad \text{ and } \quad 2^{( \eta -\frac{1}{4} - \varepsilon) j} \le \#  B_{j} \leq 2^{( \eta -\frac{1}{4} + \varepsilon) j} 
$$
for every $j \ge J$.\end{lemma}

Let us end this introduction to our model  by providing the following  concentration lemma. It states that the
non-zero coefficients are well distributed and is useful to obtain the leader large deviation spectrum of $F^d_{\alpha, \eta}$.

\begin{lemma}\label{lem:concentration}
Almost surely, for every $\varepsilon>0$, there are infinitely many
scales $j$ such that every interval of length $ 2^{- (\eta +
  \frac{1}{4}-\varepsilon)j}$ centered on dyadic numbers 
contains at most $2^{2\varepsilon j}$ non-zero
coefficients of scale $j$. 
\end{lemma}

\begin{proof}
Let us fix $\varepsilon >0$. For every dyadic interval $\lambda
\in \Lambda_{j}$, let us denote by $\lambda^{j-b}$ the dyadic interval
of scale $j-b$ that contains $\lambda$.  Remark that the random
variable that counts the number of non-zero coefficients of scale $j$ in a
interval of length $2^{-(\eta + \frac{1}{4}-\varepsilon)j}$ centered on a
dyadic interval of scale $j$ 
follows a binomial law Bin$(m,p)$ of parameters $m \leq 2^{(\frac{3}{4}- \eta+\varepsilon
  )j} $ and $p = 2^{(\eta-\frac{3}{4})j}$, so that its expectation is smaller
than $2^{\varepsilon j}$. 
Let $\Omega_{j}$ denote the event
``there is a dyadic interval $\lambda \in \Lambda_{j}$ such that for
all $b \in \{0, \dots, N-1\}$, the
interval of length $2^{-(\eta + \frac{1}{4}-\varepsilon)(j-b)}$ centered on $\lambda^{j-b}$ contains
more than $2^{2\varepsilon (j-b)}$ non-zero coefficients''. Markov
inequality leads to 
\begin{eqnarray*}
 \PP(\Omega_{j})& \le & \sum_{\lambda \in \Lambda_{j}} \prod_{b=0}^{N-1}\PP \big(\text{$\lambda^{j-b}$ contains
more than $2^{2\varepsilon (j-b)}$ non-zero coefficients}\big)\\
& \leq &  \sum_{\lambda \in \Lambda_{j}} \prod_{b=0}^{N-1}
         2^{-\varepsilon (j-2b)} \\
& \leq & 2^{j(1-\varepsilon N)} 2^{2\varepsilon N^2}
\end{eqnarray*}
which is the general term of a convergent series if $N$ is large
enough. 
\end{proof}

\subsection{Computation of  the multifractal spectrum of  the duplicated LWS}

Let us start  by examining the range of possible values for the Hölder exponent of points in the Cantor set. Obviously, since all the wavelet coefficients of scale $j$ are smaller than or equal to $2^{-\alpha j}$, the minimal regularity is greater than or equal to $\alpha$.
First, we demonstrate that in the highly lacunar case where $\eta <\frac{1}{4}$, any regularity $\beta <   \frac{2 \alpha}{4 \eta +1}$ is not observed, implying in particular the absence of points with regularity $\alpha$. 
It is due to the significant lacunarity of the series for $\eta <\frac{1}{4}$, which implies that the probability of having infinitely many intervals $\lambda$ intersecting the Cantor set with $c_{\lambda}\neq 0$   is zero.  
Furthermore, as we will prove in the subsequent lemma, with probability one, it is necessary to descend at least
  $ \frac{2 \alpha}{4 \eta +1} \log_2 \vert \lambda \vert$ scales deeper before encountering a non-zero coefficient on a  $\lambda' \subset \lambda$. 

\begin{lemma}\label{lem:minimal_regularity}
 Let $\eta<\frac{1}{4}$. Almost surely, for all $x \in \K$ one has $h_{F^d_{\alpha,\eta}}(x) \ge  \frac{2 \alpha}{4 \eta +1}.$ 
 \end{lemma}

\begin{proof}
Let us fix $h \in (\alpha,
\frac{2 \alpha}{4 \eta +1})$ and  let $\Omega_j$ denote the event 
$$
\Omega_{j} = \big\{ \exists \lambda \in R_j \text{ such that } \vert d_{\lambda_j}
\vert \ge 2^{-hj}\big\}
$$
where $R_j =\{ \lambda \in \Gamma_j :  \lambda \cap \K \neq \emptyset
\} $. For a given dyadic interval $\lambda$, there are at  $3 \cdot 2^{\ell - j}$ dyadic intervals $\lambda'$ of scale $\ell>j$ included in $3\lambda$. Furthermore, the probability that such a dyadic interval is non-zero is either $2^{(\eta-\frac{3}{4}) \ell}$ or $0$. Then, by Lemma \ref{lem:R_j}, we obtain
\begin{eqnarray*}
\PP(\Omega_j) & \le & \sum_{\lambda \in R_{j}} \PP \left( \exists \lambda'
                 \subset  3 \lambda \text{ such that }\lambda' \in \Gamma_{l} \text{
                 with }
                  j \le \ell \le \lfloor \frac{h}{\alpha}j \rfloor \text{ and  } c_{\lambda'} =
                 2^{-\alpha \ell }\right) \\
& \le & \sum_{\lambda \in R_{j}} \sum_{ j \le \ell \le \lfloor \frac{h}{\alpha}j  \rfloor } 3 \cdot 2^{\ell - j}
        2^{(\eta-\frac{3}{4}) \ell} \\
& \leq & 3 \cdot 2^{-\frac{j}{2}}  \sum_{ j \le \ell \le \lfloor \frac{h}{\alpha}j \rfloor } 
        2^{(\eta+\frac{1}{4}) \ell} \\
& \le & 3    \frac{ h}{\alpha}j 2^{-\frac{j}{2}} 2^{(\eta+\frac{1}{4})
        \frac{h}{\alpha}j}
\end{eqnarray*}
which is the general term of a convergent
series since $h < \frac{2 \alpha}{4 \eta +1}$. The Borel-Cantelli Lemma implies then that almost surely, one has
$$
\vert d_{\lambda_j} \vert \leq 2^{-hj} \quad \forall \lambda \in R_j
$$
for every $j$ large enough. For every $x \in \K$, one has $\lambda_j(x) \in R_j$ for all $j$ and the wavelet characterization of the pointwise H\"older exponent recalled in Theorem \ref{thm:waveletcharact} gives then $h_{F^d_{\alpha, \eta}} (x) \geq h$. In order to get the conclusion, one takes an increasing sequence 
 $(h_{n})_{n \in \N}$ of $(\alpha,
\frac{2 \alpha}{4 \eta +1})$ that converges to $\frac{2 \alpha}{4 \eta +1}$ to get that almost surely, for every $x \in \K$, $h_{F^d_{\alpha, \eta}} (x) \geq h_n$ for every $n \in \N$, hence the conclusion.
\end{proof}

An upper bound for the maximal regularity of the duplicated LWS is established in the following lemma, depending on whether $\eta\leq \frac{1}{4}$ or not. The bound's optimality will be demonstrated later.

\begin{lemma}\label{lem:maximal_regularity}
  Almost surely,
  there is $J \in \N$ such
  that
  $$
  d_{\lambda} \geq  \sup_{\lambda' \subset  \lambda} |c_{\lambda'}| \geq
  \begin{cases}
 2^{-\frac{\alpha}{\eta+\frac{1}{4}}(j + \log_{2}j)} & \text { if } \eta \in (\frac{1}{4},\frac{3}{4}),\\[1.5ex]
   2^{-\frac{\alpha}{2\eta}(j + \log_{2}j)} & \text{ if } \eta \in (0,\frac{1}{4}].
    \end{cases} 
$$
for every $\lambda \in R_{j}$ with  $j \geq J$. In particular, almost
surely for all $x \in \K$, one has 
$$h_{F^d_{\alpha,\eta}}(x) \leq
\begin{cases}
    \frac{\alpha}{\eta+\frac{1}{4}} & \text{ if } \eta \in (\frac{1}{4},\frac{3}{4}) ,\\[1.5ex]
   \frac{\alpha}{2\eta} & \text { if }\eta \in (0,\frac{1}{4}].
\end{cases}
$$
\end{lemma}
\begin{proof} We begin with the case $\eta \in (\frac{1}{4},\frac{3}{4})$. For $j \ge 0$, let us define the event
$$
\Omega_j = \Big\{ \exists \lambda \in R_j \, \text{ such that }  \sup_{\lambda' \subset  \lambda} |c_{\lambda'}| <2^{-\frac{\alpha}{\eta+\frac{1}{4}}(j + \log_{2}j)}\Big\}.
$$
Now, let  us fix  $j_0 =\lfloor \frac{1}{\eta +\frac{1}{4}} (j + \log_2 j)\rfloor$ so that
$2^{-\alpha j_0} \ge 2^{-\frac{\alpha}{\eta+\frac{1}{4}}(j +
  \log_{2}j)}$. Using the assumption $\eta>\frac{1}{4}$, we obtain $j < j_0 \le 2j$
for $j$ large enough. Consequently, if $\lambda \in R_j$,
all the dyadic intervals $\lambda_0 \subset  \lambda$ of scale $j_{0}$
belong to $\Gamma_{j_0}$, as observed in Remark \ref{rem:R_j}, and may then potentially have a corresponding non-zero wavelet coefficient. Hence, the number of  dyadic intervals $\lambda_{0}\subset  \lambda$ of scale $j_0$  corresponding to  a random wavelet coefficient  is equal to $2^{j_{0}-j}$. By the independence of the random variables $\xi_{\lambda_0}$ and using Lemma \ref{lem:R_j}, we obtain  
 \begin{eqnarray*}
 \PP(\Omega_{j})
 & \leq  &\sum_{ \lambda \in R_j} \PP\left(\forall \lambda_0 \subset  \lambda \text{ with } \lambda \in
          \Lambda_{j_0}, \xi_{\lambda_{j_0}}=0\right)\\
 & \leq  & 2^{\frac{j}{2}} (1-2^{(\eta-\frac{3}{4})j_0})^{2^{(j_0-j) }}\\
 & \leq &  2^{\frac{j}{2}} \exp(-2 ^{(j_0-j)}2^{(\eta-\frac{3}{4})j_0}) \\
 & \leq &  2^{\frac{j}{2}} \exp(-2^{-j +j_0(\eta +\frac{1}{4})})\\
 & \leq &C \left(\frac{\sqrt{2}}{e}\right)^{j}
 \end{eqnarray*}
for some constant $C$ and for $j$ large enough. Since the majorant is the
 general term of a convergent series, the result follows using the
 Borel-Cantelli lemma.

\smallskip

Let us now study the case $\eta \in (0,\frac{1}{4}]$. Similarly as before, we consider for every $j\geq 0 $ the event  
$$
\Omega_j=\Big\{ \exists \lambda \in R_j \, \text{ such that } \sup_{\lambda' \subset  \lambda} |c_{\lambda'}|
<2^{-\frac{\alpha}{2\eta}(j+ \log_{2}j)}\Big\}.
$$
Let $j_0 = \lfloor \frac{1}{2\eta} (j+ \log_2 j)\rfloor$. Since $\eta \leq \frac{1}{4}$, we have now $j_0 >2j$.  As previously, we need to count how many dyadic intervals  $\lambda_0 \subset  \lambda$ belong to $\Gamma_{j_0}$. At the
scale $m=2j$, all the possible dyadic intervals of size $2^{-m}$ belong to $\Gamma_m$ because $\lambda$ intersects $\K$ so is included in $\C_{\frac{j}{2}}$. After that, from the construction of the Cantor set $\K$, every dyadic interval of $\Gamma_m$ loses half of their middle length every four scales, and we find that, writing $\ell_0=j_{0}-2j$,  there remain a number of order $2^{j+\frac{3 \ell_0}{4}}= 2^{\frac{3}{4}j_0 - \frac{j}{2}}$ dyadic intervals $\lambda_{j_0}
\subset  \lambda_j$ in $\Gamma_{j_0}$. Using the same development as in the first case, we obtain
$$
\PP(\Omega_j) \leq 2^{j /2} (1-2^{(\eta-\frac{3}{4})j_0})^{2^{(\frac{3}{4}j_0-\frac{j}{2}) }}  \leq  2^{\frac{j}{2}} \exp(-2 ^{(\frac{3}{4}j_0-\frac{j}{2})}2^{(\eta-\frac{3}{4})j_0}) 
  \leq \left(\frac{{\sqrt{2}}}{e}\right)^{j},
$$
and again, the Borel-Cantelli lemma allows us to conclude. \end{proof}

By combining Lemmas \ref{lem:minimal_regularity} and \ref{lem:maximal_regularity}, we find that the possible values of the H\"older exponent of $F^d_{\alpha,\eta}$ at any point of the Cantor set lie in $[\frac{2\alpha}{4\eta +1}, \frac{\alpha}{2\eta}]$ if $\eta \in (0,\frac{1}{4}]$, and in  $[\alpha, \frac{\alpha}{\eta + \frac{1}{4}}]$ if $\eta \in (\frac{1}{4},\frac{3}{4})$.  Let us now describe the iso-H\"older sets of $F^d_{\alpha,\eta}$. We will begin by giving a covering of $\K$ using balls centered at the dyadic points associated with non-zero coefficients.

\begin{proposition}\label{prop:recouvrementK}
    Almost surely, one has
$$
\K \subset  \begin{cases}
 \displaystyle \limsup_{j \to + \infty} \bigcup_{k\, :\, \lambda_{j,k}\in A_{j}} B\left(k2^{-j},
  2^{-2\eta(1-{\varepsilon_{j}})j}\right) & \text{ if } \eta \in (0,\frac{1}{4}],\\[3ex]
 \displaystyle  \limsup_{j \to + \infty} \bigcup_{k \, : \, \lambda_{j,k}\in A_{j}} B\left(k2^{-j},
  2^{-(\eta +\frac{1}{4})(1-{\varepsilon_{j}})j}\right) & \text { if }\eta \in (\frac{1}{4},\frac{3}{4}),
  \end{cases}
$$
where we recall that  $A_j = \{ \lambda \in \Gamma_j : \, c_{\lambda} = 2^{-\alpha j} \}$ and
$${\varepsilon_{j}}= \begin{cases}
  \frac{\log_{2}j}{2\eta j}  & \text{ if } \eta \in (0,\frac{1}{4}],\\[1.5ex]
 \frac{\log_{2}j}{(\eta+\frac{1}{4}) j}  & \text { if } \eta \in (\frac{1}{4},\frac{3}{4})
 .\end{cases}$$
\end{proposition}

 \begin{proof}
   We prove the result only in the case $\eta \in (0,\frac{1}{4}]$, as the other case is similar.    If $x \in \K$, Lemma \ref{lem:maximal_regularity} implies that for every sufficiently large scale $j$, there exists $\lambda_0 \subset  \lambda_{j}(x)$ of scale $j_0\geq j$ such that $c_{\lambda_0} = 2^{-\alpha j_{0}} \geq 2^{-\frac{\alpha}{2\eta}(j+ \log_{2}j)} $. In particular, $2 \eta j_{0} \leq j + \log_{2}j \leq j + \log_{2}j_{0}$, and it follows that
 \[
 |x-k_{0}2^{-j_{0}}| <  2^{-j} \leq  2^{-2\eta j_{0} + \log_{2}j_{0}} = 2^{-2\eta(1-\varepsilon_{j_0})j_{0}}.
 \]
\end{proof}

Based on the previous result, we now introduce limsup sets that will allow to describe the iso-H\"older sets of $F^d_{\alpha,\eta}$, as classically done (see e.g. \cite{Jaffard:00b,Jaffard:02,SaesSeuret}). We reproduce the following proof for the sake of completeness.

\begin{lemma}\label{lem:Edelta}
    For every $\delta \in (0,1]$, let us consider the random sets
\[
E_{\delta} = \limsup_{j \rightarrow + \infty}\bigcup_{k \, :\,  \lambda_{j,k}\in A_{j}}B\left( k 2^{-j}, 2^{-\delta (1-\varepsilon_{j})j} \right),
\]
where  $A_j = \{ \lambda \in \Gamma_j : \, c_{\lambda} = 2^{-\alpha j} \}$ and the sequence $(\varepsilon_j)_{j \in \N}$ is defined as in Proposition \ref{prop:recouvrementK}. 
\begin{itemize}
    \item If $x \in E_{\delta}$, then $h_{F^d_{\alpha, \eta}} (x) \leq \frac{\alpha}{\delta}$.
    \item If $h_{F^d_{\alpha, \eta}} (x) < \frac{\alpha}{\delta}$, then $x \in E_{\delta}$.
\end{itemize}
\end{lemma}

\begin{proof}
Let us first assume that $x \in E_{\delta}$. Then there exist infinitely many scales $j$ such that $x$ belongs to $B\left( k 2^{-j}, 2^{-\delta (1-\varepsilon_{j})j} \right)$ with $c_{j,k} = 2^{-\alpha j}$. If $j_0 = \lfloor \delta (1-\varepsilon_j)j\rfloor$, we can directly observe that $\lambda_{j,k} \subseteq \lambda_{j_0}(x)$, hence
$
d_{\lambda_{j_0}(x)} \geq 2^{-\alpha j}.
$
We get the conclusion by Theorem \ref{thm:waveletcharact}, 
since $\varepsilon_j \to 0$ as $j$ tends to $+ \infty$. 

Furthermore, if $h_{F^d_{\alpha, \eta}} (x) < \frac{\alpha}{\delta}$, a second application of Theorem \ref{thm:waveletcharact} implies the existence of infinitely many scales $j$ for which
$
d_j(x) > 2^{-\frac{\alpha}{\delta}j}.
$
It means that there is at least a dyadic interval $\lambda' \subseteq 3\lambda$ of scale $j' $ with $j \leq j' < \frac{1}{\delta}j$ with $c_{\lambda' }=c_{j',k'} =2^{-\alpha j'} $. Since $k'2^{-j'} \in 3 \lambda$, one gets
$$
|x-k'2^{-j'}|< 2^{1-j}< 2^{1-\delta j'}<2^{-\delta (1-\varepsilon_{j'})j'}
$$
since $\varepsilon_{j'}j' >1$ if $j'$ is large enough. Therefore $x \in E_{\delta}$.
\end{proof} 

\begin{remark}\label{rem:lemmaEdelta}
The preceding Lemma is a classical result that applies to any wavelet series with wavelet coefficients at scale $j$ being either $2^{-\alpha j}$ or $0$.
\end{remark}

\begin{proposition}
Let
\[
G_{\delta} := \K \cap\left( \bigcap_{0<\delta' < \delta} E_{\delta'}
\setminus \bigcup_{\delta < \delta' \leq 1} E_{ \delta'}\right) \ \text{ if }
\ \delta <1  \quad \text{ and }   \quad G_1 :=\K \cap \bigcap_{0<\delta' < 1} E_{\delta'}.
\]
For every $\delta \in (0,1]$, one has
$$G_{\delta} = \left\{x \in [0,1] : h_{F^d_{\alpha,\eta}}(x) = \frac{\alpha}{\delta}\right\}.$$
\end{proposition}

\begin{proof}
It suffices to use Lemma \ref{lem:Edelta} and to observe that the points $\in [0,1] $ for which $h_{F^d_{\alpha,\eta}}(x)<+ \infty$ are located on $\mathcal{K}$.
\end{proof}

Now, it remains to compute the Hausdorff dimension of the sets $G_{\delta}$ for $\delta$ in $[2\eta, \frac{4\eta + 1}{2}]$ if $\eta \in (0,\frac{1}{4}]$, and in $[ \eta + \frac{1}{4},1]$ if $\eta \in (\frac{1}{4},\frac{3}{4})$.
Note that the unions and the intersections  appearing in
the definition of $G_{\delta}$ can be taken countable by considering
subsequences converging to $\delta$. For this reason, in what follows,
everything can be made countable by fixing 
a dense sequence $(\delta'_{n})_{n \in \N}$ of $[0,1]$ and estimate
the Hausdorff dimension of each $E_{\delta'}$.

The upper bounds for $\dim_{\mathcal{H}}(E_{\delta}\cap \K)$ and hence for $\dim_{\mathcal{H}}(G_{\delta})$ can be obtained from Lemma \ref{lem:cardFj}, which provides the cardinality of the sets $A_{j}$.

\begin{proposition}\label{prop:bornesup}
    \begin{enumerate}
\item If $\delta \in (0, \frac{1}{2}]$, then almost surely
$$
\dim_{\mathcal{H}}(E_{\delta}\cap \K) \leq \frac{\eta}{\delta} 
$$
and $\mathcal{H}^{\frac{\eta}{\delta}  } (E_{\delta'}\cap \K) = 0$ for all $\delta'> \delta$.
\item If $\delta \in ( \frac{1}{2},1]$, then almost surely
$$
\dim_{\mathcal{H}}(E_{\delta}\cap \K) \le \frac{\eta+\frac{1}{4}}{\delta} -\frac{1}{2} 
$$
and $\mathcal{H}^{\frac{\eta+\frac{1}{4}}{\delta} -\frac{1}{2}  } (E_{\delta'}\cap \K) = 0$ for all $\delta'> \delta$.
\end{enumerate}
\end{proposition}

\begin{proof}
For every $J \in \N$, it is clear that the set
$$
\bigcup_{j \geq J }\bigcup_{k \, :\,  \lambda_{j,k}\in A_{j}} B\left( k 2^{-j}, 2^{-\delta (1-\varepsilon_{j})j} \right) \cap \K
$$
 forms a covering of $G_{\delta}$. When intersecting with $\K$, for a
fixed $j \geq J$,
we need to count the number $N_{\delta,j}$ of non-zero coefficients associated with
dyadic intervals $\lambda$ which are contained in $\C_{\lfloor \frac{j}{4} \rfloor}$ and are within a distance
less than $2^{-\delta (1-\varepsilon_{j}) j}$ from the
set $\C_{\lfloor \frac{j}{2}\rfloor}$. It will then suffice to study the convergence of the series
$$
 \sum_{j \geq J } N_{\delta, j}  2^{-\delta sj}.
$$
\begin{itemize}
\item If $\delta \leq \frac{1}{2}$, then $2^{-\delta(1-\varepsilon_{j}) j} \geq
  2^{-2n}$ for $j=4n$. Hence, the considered intervals within $\C_{\lfloor \frac{j}{4}
    \rfloor}$ and  at a distance less than $2^{-\delta(1-\varepsilon_{j})
     j}$ of  $\C_{\lfloor \frac{j}{2} \rfloor}$ are exactly the dyadic intervals of  $\C_{\lfloor \frac{j}{4}
    \rfloor}$. For every $\varepsilon>0$, Lemma \ref{lem:cardFj}  implies that, on an event of probability one that does not depend on $\delta$, one has
  $N_{\delta ,j} \leq 2^{(\eta + \varepsilon )j} $ for $j$ large enough, which implies
$$
 \sum_{j \geq J } N_{\delta, j}  2^{-\delta'(1-\varepsilon_{j}) sj}  \leq \sum_{j \geq J}  2^{(\eta + \varepsilon)j} 2^{-\delta(1-\varepsilon_{j}) s j} < + \infty
$$
if $s > \frac{\eta + \varepsilon}{\delta(1-\varepsilon)}$, since $\varepsilon_j \leq \varepsilon$ for $j$ large enough. Consequently, almost surely, one has $\dim_\mathcal{H} (E_{\delta}\cap \K) \leq s$. By replacing $\varepsilon$ by a sequence that converges to $0$, we conclude that $\dim_{\mathcal{H}}(E_{\delta}\cap \K)\le \frac{\eta}{\delta}$  almost surely.

\item If $\delta> \frac{1}{2}$,  we have to consider 
 the intervals $\lambda \in A_{j}$ 
 such that $\lambda$ is at a distance
 less than $2^{-\delta(1-\varepsilon_j) j}$ of the 
 set $\C_{\lfloor \frac{j}{2}\rfloor}$. Hence, we have first to count the number of
 dyadic interval $\lambda$ of scale $j$ which belong to $\C_{\lfloor \frac{j}{4} \rfloor}$
 and at a distance less than $2^{-\delta (1-\varepsilon_j)j}$ of $\C_{l}$ where $l$ is of order $\delta(1-\varepsilon_j) \frac{j}{2}$. Indeed, with this choice of $l$, the set $\C_l$ is formed by
 intervals of length of order  $2^{-\delta(1-\varepsilon_j) j}$.
This  number is bounded by $C2^{j-l}=C 2^{(1-\frac{\delta(1-\varepsilon_j)}{2})j}$ for some constant
 $C$ independent of $j$. Using  Markov inequality and the Borel
 Cantelli lemma as in Lemma \ref{lem:cardFj}, we get that for every $\varepsilon>0$, almost surely, 
 $N_{\delta,j} \leq 2^{(\frac{1}{4}+ \eta-\frac{\delta(1-\varepsilon_j)}{2}+ \varepsilon)j}$ if $j$ is large enough. It follows that 
$$
 \sum_{j \geq J } N_{\delta, j} 2^{-\delta sj} \leq   \sum_{j \geq J}2^{(\frac{1}{4}+ \eta-\frac{\delta(1-\varepsilon_j)}{2} + \varepsilon)j} 2^{-\delta(1-\varepsilon_j)s j} < + \infty
$$
if $s > \frac{\frac{1}{4}+ \eta + \varepsilon}{\delta (1-\varepsilon)} - \frac{1}{2}$. As in this first case, it
follows  $ \dim_{\mathcal{H}}(E_{\delta}\cap \K) \leq \frac{\frac{1}{4}+ \eta
}{\delta } - \frac{1}{2}$ on an event of probability one.

\end{itemize}
Combining both cases yields the stated upper bounds.  
\medskip

For the second part, observe that if $\delta'> \delta$, then $\dim_{\mathcal{H}}(E_{\delta'} \cap \K)$ is strictly smaller than the upper bound obtained for $\dim_{\mathcal{H}}(E_\delta \cap \K)$ from the previous part of the proof. Therefore, the conclusion follows.
\end{proof}

Establishing a lower bound for the Hausdorff dimension of $G_{\delta}$ requires the use of ubiquity arguments. We will employ a simplified version of the following result from \cite{Beresn:06} for our specific application.

\begin{theorem}[General mass transference principle]\cite{Beresn:06}\label{thm_transference}
Let $X$ be a compact set in $\R^n$ and assume that there exist $s\leq n$ and $a,b,r_0>0$ such that
\begin{equation}\label{eqGMT}
a r^s \leq \mathcal{H}^{s}(B \cap X) \leq b r^s
\end{equation}
for any ball $B$ of center $x \in X$ and of radius $r \leq r_0$. Let $s' >0$. Given a ball $B = B(x,r)$ with center in $X$, we set
\[
B^{s'} = B \left(x, r^{\frac{s'}{s}} \right).
\]
Assume that $(B_n)_{n \in \N}$ is a sequence of balls with center in $X$ and radius $r_n$  such that the sequence $(r_n)_{n \in \N}$ converges to 0. If 
\[
\mathcal{H}^s \left(  X \cap \limsup_{n \rightarrow + \infty} B_n^{s'} \right) = \mathcal{H}^s( X),
\]
then
\[
\mathcal{H}^{s'} \left(X \cap \limsup_{n \rightarrow + \infty} B_n \right) = \mathcal{H}^{s'}(X).
\]
\end{theorem}

The following result provides the Hausdorff dimension of the set $G_\delta$.

\begin{proposition}\label{prop_sup}
\begin{enumerate}
\item If $\eta \in [\frac{1}{4}, \frac{3}{4})$, then almost surely
$$
\dim_{\mathcal{H}}(G_{\delta}) = 
 \frac{\eta+\frac{1}{4}}{\delta} -\frac{1}{2}  
$$
for every $\delta \in [\eta +\frac{1}{4},1]$. 
\item If $\eta \in (0, \frac{1}{4}]$, then almost surely
$$
\dim_{\mathcal{H}}(G_{\delta}) = \begin{cases}
\frac{ \eta}{\delta} & \text{ if }  \delta \in [2\eta, \frac{1}{2}],\\[1ex]
\frac{\eta+\frac{1}{4}}{\delta} -\frac{1}{2} & \text{ if } \delta
\in [\frac{1}{2}, \frac{4\eta +1}{2}].\\
\end{cases}
$$
\end{enumerate}
\end{proposition}

\begin{proof}
The upper bounds are directly derived from Proposition \ref{prop:bornesup} since $G_\delta \subseteq E_{\delta'}\cap \K$ for every $\delta'> \delta$. Let us focus on establishing the lower bounds. First, let us estimate $\dim_{\mathcal{H}}(E_{\delta})$ from below.
 
Let us start by assuming that either $\eta> \frac{1}{4}$ and  $\delta \geq \eta + \frac{1}{4}$,
 or $\eta \leq \frac{1}{4}$ and $\delta > \frac{1}{2}$. In both cases, $\delta
\geq \frac{1}{2}$.  
As done in the proof of Proposition \ref{prop:bornesup}, it suffices to consider in the definition of $E_{\delta}$ the non-zero coefficients $c_{\lambda}$ where $\lambda$ belongs to $\C_{\lfloor \frac{j}{4}   \rfloor}$ and is at a distance at most $2^{-\delta(1-\varepsilon_j) j}$ from $\C_{\lfloor \frac{j}{2} \rfloor} $.
Let us fix $\varepsilon>0$ and a dyadic interval $\lambda_0$ of scale $J=2l$ appearing in $\mathcal{C}_l$ where $l= \delta (1-\varepsilon_j)\frac{j}{2}$. Notice that 
$$ \mathbb{P}\big( \exists \lambda \in \Gamma_j \, , \lambda \subseteq \lambda_0, \, \text{such that } c_{\lambda} =2^{-\alpha j} \big)
     = 2^{(\eta - \frac{3}{4})j} 2^{j-2l}
     = 2^{\big(\frac{ \eta+ \frac{1}{4}}{\delta}(1-\varepsilon_j)  -1\big) J}
    \geq 2^{\big(\frac{ \eta+ \frac{1}{4}}{\delta}(1-\varepsilon)  -1\big) J}
$$
if $j$ is large enough. 
Consider a new (classical) lacunary wavelet series $\widetilde{F}$ on $\mathcal{K}$ where 
$$
\widetilde{c}_{J,k}= \begin{cases}
2^{-\alpha J} \widetilde{\xi_{J,k}} & \text{ if }  \lambda_{J,k} \text{ appears in } \mathcal{C}_{\lfloor\frac{J}{2} \rfloor}\\
0 & \text{ otherwise,}
\end{cases}
$$
 and 
where  $ (\widetilde{\xi_{J,k}})_{J,k}$ denotes a sequence of independent random Bernoulli
variables of parameter $2^{(\widetilde{\eta}-\frac{1}{2})J}$
with
$$
\widetilde{\eta} =  \frac{\eta+ \frac{1}{4}}{\delta}(1-\varepsilon) -\frac{1}{2}
$$
Now, remark that the position of the non-zero coefficients we consider in  $E_{\delta}$ can be seen as the position  of the non-zero coefficients of $\widetilde{F}$. The Hausdorff dimension of $E_{\delta}$ is then larger than the  Hausdorff dimension of the set of minimal regularity of
 this new lacunary wavelet series. It follows  from Proposition
 \ref{prop:lowerboundclassic} that 
$$
\dim_{\mathcal{H}}E_{\delta} \geq \widetilde{\eta} =  \frac{\eta+ \frac{1}{4}}{\delta}(1-\varepsilon) -\frac{1}{2}
.$$

   \medskip

Let us now focus on the case $\delta \in [2 \eta,\frac{1}{2}]$ which only occurs
 for $ \eta \leq \frac{1}{4}$. As in the classical case, we use an ubiquity
argument. Note that the argument of the general mass transference
could not have been applied in the case we just dealt for large values
of $\delta$ for the following reason. The assumptions of Theorem
\ref{thm_transference} require that the balls are centered in
$X=\K$. For $\delta>\frac{1}{2}$, if $k 2^{-j }\notin \K$, the ball $B(k2^{-j},
2^{-\delta(1-\varepsilon_j) j})$ of $E_{\delta}$ does not necessarily meet the Cantor set $\K$ (even by multiplying the radius with a constant independent of $j$). At the opposite, if $\delta<\frac{1}{2}$, all balls  $B(k2^{-j},
2^{-\delta(1-\varepsilon_j) j})$  of $E_{\delta}$ intersect the Cantor $\K$, and by doubling it we can suppose that each ball of $E_{\delta}$ is centered in $\K$.

Applying Theorem \ref{thm_transference} thanks to Lemma \ref{lem:maximal_regularity}, we get
\[
\mathcal{H}^{\eta/\delta} \left(  \K \cap \limsup_{j \to + \infty}
  \bigcup_{k\, : \, \lambda_{j,k} \in A_j}B(k2^{-j},  2^{-\delta(1-{\varepsilon_{j}})j}) \right) =  \mathcal{H}^{\eta/\delta}(\K)>0
\]
since $\frac{\eta}{\delta}\leq \frac{1}{2}$. We conclude that almost surely, $\dim_{\mathcal{H}}(E_{\delta}) \ge \frac{\eta}{\delta}$.
\medskip

Let us now compute the dimension of $G_{\delta}$. Notice that the unions and intersections in the definition of $G_{\delta}$ can be made countable by considering subsequences converging to $\delta$, so that we can consider that we work on an event of probability one independent of $\delta$. In the case where $\delta \le \frac{1}{2}$, we have
\[
\mathcal{H}^{\frac{\eta}{\delta}} (G_{\delta}) = \mathcal{H}^{\frac{\eta}{\delta}} \left( \bigcap_{\delta'  < \delta} E_{\delta'} \cap \K \right) - \mathcal{H}^{\frac{\eta}{\delta}} \left( \bigcup_{\delta'  > \delta} E_{\delta'}\cap \K \right) =  \mathcal{H}^{\frac{\eta}{\delta}} \left( \bigcap_{\delta'  < \delta} E_{\delta'} \cap \K \right)
\]
since the $\frac{\eta}{\delta}$-dimensional Hausdorff measure of $E_{\delta'}\cap \mathcal{K}$ vanishes
if $\delta' > \delta$ using Proposition \ref{prop:bornesup}. Therefore,
\[
\mathcal{H}^{\frac{\eta}{\delta}} (G_{\delta}) = \mathcal{H}^{\frac{\eta}{\delta}} \left( \bigcap_{\delta'  < \delta} E_{\delta'} \cap \mathcal{K} \right) \geq \mathcal{H}^{\frac{\eta}{\delta}} (E_{\delta} \cap \mathcal{K}) >0,
\]
which implies $\dim_{\mathcal{H}}(G_{\delta}) \geq 
\frac{\eta}{\delta}$. 
The case where $\delta > \frac{1}{2}$ can be established similarly.
\end{proof}

\bigskip

\subsection{Computation of the leader large deviation spectrum of  the duplicated LWS}

Let us explain the idea of the computation of the leader large deviation
spectrum. First, recall that it can be equivalently be computed using the restricted wavelet leaders. Furthermore, essentially, a non-zero coefficient $c_{j,k} = 2^{-\alpha j}$ generates a restricted wavelet leader of size $2^{-hj'}$ where $j' = \frac{h}{\alpha}j$. Given that the non-zero coefficients are well spread on the Cantor set $\K$ according to Lemma \ref{lem:concentration}, approximately $\# A_{j}$ restricted wavelet leaders of order $2^{-hj'}$ are expected at scale $j'$. Consequently, Lemma \ref{lem:cardFj} implies $\rho_{F^d_{\alpha,\eta}}(h) =  \frac{h \eta }{\alpha}$. Particularly, $\rho_{F^d_{\alpha,\eta}}(\alpha) = \eta$ as anticipated.

Note that the possible values for $h$ are already known. Indeed,
from Lemma \ref{lem:maximal_regularity}, we know that
almost surely, the restricted wavelet leaders associated with dyadic intervals
$\lambda  \in R_{j}$ satisfy 
$$
e_{\lambda} \geq \begin{cases} 2^{-(\frac{\alpha}{2\eta}+
    \varepsilon)j} & \text{ if } \eta \in (0,\frac{1}{4}]\\[1ex]
2^{-(\frac{\alpha}{\eta+\frac{1}{4}}+\varepsilon)j} & \text { if }\eta \in (\frac{1}{4},\frac{3}{4}) \end{cases}
$$
for $j$ large enough. Note also that dyadic intervals $\lambda  \in \Gamma_{j}\setminus R_{j}$ contain a smaller number of dyadic subintervals corresponding to the random coefficients appearing in the construction of $F^d_{\alpha,\eta}$ than dyadic intervals in $R_{j}$. In particular, for each $\lambda  \in \Gamma_{j}\setminus R_{j}$, there is a scale $j'$ such that $c_{\lambda'}=0$ if $\lambda' \subset \lambda$ is a dyadic interval of scale $j'$. 
This implies that the corresponding non-zero restricted wavelet leaders cannot be smaller than the restricted wavelet leaders associated with intervals in $R_j$ (since smaller restricted wavelet leaders are realized at larger scales).

More precisely, in the case  $\eta \in (\frac{1}{4},\frac{3}{4})$, the same arguments as those of the proof of
Lemma \ref{lem:maximal_regularity} give that almost surely, for the dyadic intervals $\lambda  \in \Gamma_{j}\setminus R_{j}$ and for every $\varepsilon >0$, if $e_{\lambda }\neq 0$, then  $e_{\lambda}\geq 2^{-(\frac{\alpha}{\eta+\frac{1}{4}}+\varepsilon)j}$ for sufficiently large $j$. Hence, the support of the leader large deviation spectrum lies within $[\alpha, \frac{\alpha}{\eta+\frac{1}{4}}]$.

In the highly lacunar case $\eta \in (0,\frac{1}{4})$, we know that if $\lambda' \subset  \lambda$  is a dyadic interval of scale  $j' \geq 2j$ with $\lambda  \in \Gamma_{j}\setminus R_{j}$, then $\lambda' \notin \Gamma_{j'}$. Consequently, the restricted wavelet leader $e_{\lambda}$ is either equal to $0$ or $2^{-\alpha j'}$ with $j' < 2j$, in which case $e_{\lambda} \geq 2^{-2 \alpha j} \geq
2^{-\frac{\alpha}{2\eta}j}$. This implies that the support of the leader large deviation spectrum is contained within $[\alpha, \frac{\alpha}{2\eta}]$.

\begin{proposition}
  Almost surely, one has
  $$
\rho_{F^d_{\alpha,\eta}}(h) = \begin{cases}
\frac{h \eta}{\alpha}& \text{ if } h  \in [\alpha, h_{\max}],\\[1ex]
1 & \text{ if } h= + \infty,\\[1ex]
- \infty & \text{otherwise.}
\end{cases}
$$
where $$h_{\max} = \begin{cases}\frac{\alpha}{2 \eta} & \text{ if
  } \eta \in (0,\frac{1}{4}],\\[1ex]
  \frac{\alpha}{\eta +\frac{1}{4}} & \text{ if } \eta \in (\frac{1}{4},\frac{3}{4}).
  \end{cases}$$
\end{proposition}

\begin{proof}
The result is straightforward if $h \notin [\alpha, h_{\max}]$. Let us then fix $h \in [\alpha, h_{\max}]$ and $\varepsilon >0$. For every sufficiently large scale $j$, one has almost surely
$$\#\{ \lambda \in \Lambda_{j } : 2^{- (h+\varepsilon)j} \leq e_{\lambda} \leq 2^{- (h-\varepsilon)j} \} \leq \sum_{j' = \lfloor \frac{h-\varepsilon}{\alpha} j\rfloor}^{\lfloor \frac{h+\varepsilon}{\alpha} j\rfloor +1} \#A_{j'} \leq C j 2^{(\eta + \varepsilon)\frac{h+\varepsilon}{\alpha} j}$$
for some constant $C>0$ and sufficiently large $j$, where we have used Lemma \ref{lem:cardFj}. This directly implies the upper bound for $\rho_{F^d_{\alpha,\eta}}(h)$.

The lower bound holds in the case where $\eta \in (0,\frac{1}{4}]$ and $h \in [2 \alpha, \frac{\alpha}{2 \eta}]$ directly from Propositions \ref{prop:comparison} and \ref{prop_sup}. Thus, we can assume either $\eta \in (0,\frac{1}{4}]$ and $h < 2 \alpha$, or $\eta \in (\frac{1}{4},\frac{3}{4})$. Let $j'$ be a scale such that every interval of length $2^{-(\eta + \frac{1}{4}-\varepsilon)j'}$ contains at most $2^{2\varepsilon j'}$ non-zero coefficients. It is known from Lemma \ref{lem:concentration} that almost surely there are infinitely many such scales. Given $j = \lfloor \frac{\alpha}{h+\varepsilon}j'\rfloor$, it can be directly computed that $2^{-j}\leq 2^{-(\eta + \frac{1}{4}-\varepsilon)j'}$. This implies that every $\lambda \in \Lambda_{j}$ contains at most $2^{2\varepsilon j'}$ non-zero coefficients of scale $j'$. By applying Lemma \ref{lem:cardFj} again, we obtain
$$
\#\{ \lambda \in \Lambda_{j} : e_{\lambda} \geq 2^{- (h+\varepsilon)j} \} \geq \# A_{j'} 2^{-2\varepsilon j'} \geq 2^{(\eta - 3 \varepsilon )\frac{h+ \varepsilon}{\alpha}j}.
$$
This inequality implies that almost surely
$$
\rho_{{\rm cum}, F^d_{\alpha,\eta}}(h) \geq \frac{h \eta}{\alpha}
$$
for every $h \in [\alpha, h_{\max}]$ if $\eta \in [\frac{1}{4},\frac{3}{4})$, and for every $h \in [\alpha, 2 \alpha)$ if $\eta \in (0,\frac{1}{4}]$. Equation \eqref{eq:rhocum} leads to the desired conclusion.
\end{proof}

\section{Detection of overestimations of multifractal spectra}

This section aims at proving Theorem \ref{TheoDetect} and its Corollary \ref{cor:validity}. 
Recall that we consider a decreasing sequence of closed subsets $\I_j, j \in \N$, of $[0,1]$. Let us set 
$$
\I = \bigcap_{j \in \N}\I_j \quad \text{ and } \quad  H = \dim_{\mathcal{H}}(\I).  
$$
Fix also $\alpha>0$. 
The wavelet series we study is defined as 
$$
F^{(\alpha, (\I_j))}= \sum_{j \in \N} \sum_{k=0}^{2^j-1} c_{j,k} \psi_{j,k}
\quad \text{with} \quad
c_{\lambda} = \begin{cases} 2^{- \alpha j} & \text{ if } \lambda_{j,k} \cap \I_j \neq \emptyset ,\\
0 &\text{ otherwise }. \end{cases}
$$
The following lemma gives an estimation of the cardinality of the set of non-zero coefficients of $F^{(\alpha, (\I_j))}$ at scale $j$. 

\begin{lemma}\label{lem:Gamma_j}
For every $j \in \N$, let $\Gamma_j$ be the set defined as 
$$\Gamma_j = \{ \lambda \in \Lambda_j :  c_{\lambda} \neq 0 \}.$$
Assume that $F^{(\alpha, (\I_j))}$ satisfies the leader large deviation formalism. 
    Then, for every $\varepsilon>0$, there exists $J \in \N$ such that  
    $$
    2^{(H- \varepsilon)j} \le \# \Gamma_j \le 2^{(H+\varepsilon)j}
    $$
    for all $j \geq J$. 
\end{lemma}

\begin{proof}
Assume first that $H<1$. Clearly, one has $$\mathcal{D}_{F^{(\alpha, (\I_j))}}(h) =  \begin{cases} H & \text{ if } h = \alpha ,
\\ 1 &\text{ if } h = + \infty ,
\\ - \infty & \text{ otherwise.} \end{cases}$$
Hence, since $F^{(\alpha, (\I_j))}$ satisfies the leader large deviation formalism, we have $\log_2 \# \Gamma_j \le \rho_{F^{(\alpha, (\I_j))}}(\alpha) + \varepsilon$, {\it i.e.} $\#\Gamma_j \le 2^{j(H+\varepsilon)}$ for $j$ large enough. 

The lower bound is proved by contradiction. Assume that there is a sequence $(j_n)_n$ such that $\# \Gamma_{j_n} \le 2^{j_n(H-\varepsilon)}$. The dyadic intervals of $\Gamma_{j_n}$ form a covering of
$\I$ by less than $2^{j_n(H-\varepsilon)}$ intervals of size $2^{-j_n}$. It follows that  $\dim_{\mathcal{H}}(\I) \le H- \varepsilon$, which gives the contradiction.

If $H=1$, the upper bound is direct and the lower bound is obtained in the same way.
\end{proof}

\noindent{\bf{Proof of Part 1 of Theorem \ref{TheoDetect}}}. Assume first that $H<1$. 
Let us fix $\varepsilon >0$ and $\beta>0$. Our aim is to show that the set of ``normal duplication''
$$
ND(j, \beta, \varepsilon) = \big\{ \lambda \in \G_j : 2^{(\beta H  - 4 \max(1,\beta) \varepsilon)j}\le  \# \{ \lambda' \subset \lambda :  \lambda' \in \G_{\lfloor (1+ \beta) j \rfloor} \}\le 2^{ (\beta H+ 4 \max(1,\beta) \varepsilon)j} \big\}
$$
satisfies
$$
2^{(H-2\varepsilon)j} \le \# ND(j, \beta, \varepsilon)\le 2^{(H+\varepsilon)j}
$$
for large values of $j$. The upper bound of $\# ND(j, \beta, \varepsilon)$ is easily derived from Lemma \ref{lem:Gamma_j}. We proceed by contradiction to establish the lower bound. Suppose, then, that there exists a sequence $(j_n)_n$ such that, for all  $n$,  
\begin{equation}\label{eq:ND}
    \# ND(j_n, \beta, \varepsilon) \le 2^{(H- 2\varepsilon)j_n}. 
\end{equation}
In the same spirit as $ ND(j, \beta, \varepsilon)$, we introduce the set of dyadic intervals of scale $j$ with ``slow duplication'' at scale $\lfloor (1+\beta)j \rfloor$
$$
SD(j,\beta, \varepsilon) =   \big\{ \lambda \in \G_j :   \# \{ \lambda' \subset \lambda:\lambda' \in \G_{\lfloor (1+ \beta) j \rfloor} \}\le 2^{ (\beta H- 4 \max(1,\beta)\varepsilon)j} \big\}
$$
and the set of dyadic intervals with ``fast duplication'' (which is empty in the case $H=1$)
$$
FD(j,\beta, \varepsilon) =  \big\{ \lambda \in \G_j :  \# \{ \lambda' \subset \lambda : \lambda' \in \G_{\lfloor (1+ \beta) j \rfloor} \}\ge 2^{ (\beta H + 4\max(1,\beta) \varepsilon)j} \big\}.
$$
Clearly, one has
$$\G_j = ND(j,\beta, \varepsilon) \cup SD(j,,\beta, \varepsilon) \cup FD(j,\beta, \varepsilon).$$
First, since we know from Lemma \ref{lem:Gamma_j} that $\# \G_{\lfloor (1+ \beta) j \rfloor}$ is bounded by $2^{(1+\beta)(H+\varepsilon)j}$ for $j$ large enough, we obtain that
$$
\# FD(j, \beta, \varepsilon) \times 2^{(\beta H+4 \max(1,\beta) \varepsilon)j} \le \# \G_{\lfloor (1+ \beta) j \rfloor} \le 2^{(1+\beta)(H+\varepsilon)j}
$$
for $j$ large enough. It  comes that 
\begin{equation}\label{eq:FD}
\# FD(j,\beta, \varepsilon) \le 2^{(H+(1+\beta) \varepsilon - 4 \max(1,\beta) \varepsilon)j} \le 2^{(H-2\varepsilon)j}. 
\end{equation}
In order to estimate the cardinality of $SD(j,\beta, \varepsilon)$, let us introduce the set of the ``children'' of $SD(j, \beta, \varepsilon)$ by setting
$$
CSD(j, \beta, \varepsilon) = \{ \lambda' \in \G_{\lfloor (1+ \beta) j \rfloor} : \lambda' \subset \lambda \text{ with } \lambda \in SD(j, \beta, \varepsilon) \}.
$$
Again, Lemma  \ref{lem:Gamma_j} implies that $\# SD(j, \beta, \varepsilon)$ is bounded by $2^{j(H+\varepsilon)}$ for $j$ large enough, and since the dyadic intervals of $SD(j, \beta, \varepsilon)$ present a slow duplication, we can control the cardinal of $CSD(j, \beta, \varepsilon)$ by
\begin{equation}\label{eq:CSD}
    \# CSD(j, \beta, \varepsilon) \le 2^{(H+\varepsilon)j} 2^{ (H\beta - 4 \max(1,\beta) \varepsilon) j} \le 2^{((1+\beta) H-3 \max(1,\beta) \varepsilon )j}.
\end{equation}
For each $n \ge 0$, the daydic intervals of $$
ND(j_n, \beta, \varepsilon) \cup FD(j_n, \beta, \varepsilon) \cup CSD(j_n, \beta, \varepsilon)$$ 
form a covering of $\I$ with sets of diameter smaller than $2^{-j_n}$.
Together with \eqref{eq:ND}, \eqref{eq:FD}  and \eqref{eq:CSD}, it implies that, for any $r>0$ and $0 <s<1$, 
$$
{\mathcal{H}}^s_r(\I) \le 2 \times 2^{j_n(H-2\varepsilon)} 2^{-j_n s}+ 2^{j_n((1+\beta) H-3 \max(1,\beta) \varepsilon)} 2^{-(1+\beta)j_n s}
$$
for any $j_n$ such that $2^{-j_n} \le r$. Taking $s = H - \varepsilon$, we obtain
$$
{\mathcal{H}}^s_r(\I) \le 2 \times 2^{-\varepsilon j_n} + 2^{-j_n \max(1,\beta) \varepsilon}.
$$
It follows that $\lim_{r \to 0^+} {\mathcal{H}}^s_r(\I) =0$ and therefore $\text{dim}_H(\I) \le s- \varepsilon$, which is impossible. It gives the conclusion. 

If $H=1$, the result is obtained in a similar way, with obvious simplifications (some upper bounds are immediate, and the set $FD(j, \beta, \varepsilon)$ is empty). 
\hfill \qed

\bigskip

\noindent{\bf Proof of Part 2 of Theorem \ref{TheoDetect}} Let $\eta \in (1-H,1]$. We consider the lacunarized version of $F^{(\alpha, (\I_j))}$ given by 
$$
L_{\eta}(F^{(\alpha, (\I_j))})= \sum_{j \in \N} \sum_{k=0}^{2^j-1} \xi^{\eta}_{j,k}c_{j,k} \psi_{j,k}
\quad \text{with} \quad
c_{j,k} = \begin{cases} 2^{- \alpha j} & \text{ if } \lambda_{j,k} \cap \I_j \neq \emptyset , \\
0 &\text{ otherwise},\end{cases}
$$
and where $(\xi^{\eta}_{j,k})_{j,k}$ denotes a sequence of independent random Bernoulli variables of parameter $2^{(\eta-1)j}$.  Let us begin by noting that, from the decreasing property of the sequence $(\I_j)_{j \in \N}$, one can easily establish that
$$
\big\{\lambda \in \Lambda_j : \sup_{\lambda' \subset \lambda} |\xi^{\eta}_{\lambda'}c_{\lambda'} |  \geq 2^{-(h+\varepsilon)j } \big\}  \subset  \G_j 
$$
for all $h >0$ and all $\varepsilon>0$. Lemma \ref{lem:Gamma_j} directly implies that
$$
\rho_{L_{\eta}(F^{(\alpha, (\I_j))})}(h) \leq  H
$$
for all $h >0$. Now, let us show that if
$$\tau = \eta + H-1,
$$
then
$$
\rho_{L_{\eta}(F^{(\alpha, (\I_j))})}\left( \frac{\alpha H}{\tau}\right) =  H
$$
which will yield the desired result.

Notice that, using Lemma \ref{lem:Gamma_j}, the definition of $\tau$ ensures that there are approximately  $2^{\tau j}$  non-zero wavelet coefficients expected at the scale $j$ in the expansion of $L_{\eta}(F^{(\alpha, (\I_j))})$ for $j$ large enough. 
Fix an arbitrary $n \in \N_0$ and consider $$\beta_n = \frac{H}{\tau} + \frac{1}{n} -1 \quad \text{and} \quad
\varepsilon_n = \frac{ \tau}{8 n \max(1,\beta_n)}.$$
The first part of the theorem gives the existence of a $J_n \in \N$ such that, for all $j \ge J_n$ 
$$
2^{(H- 2\varepsilon_n )j} \le \# ND(j,\beta_n, \varepsilon_n) \le 2^{(H+\varepsilon_n)j}.
$$
For every $j \ge J_n$, we consider the event
$$
\Omega_{j,n} = \big\{ \exists \lambda \in ND(j,\beta_n,\varepsilon_n) \text{ such that } \sup_{\lambda' \subset \lambda} |\xi^{\eta}_{\lambda'}c_{\lambda'} | < 2^{-\alpha (1+ \beta_n) j}\big\}.
$$
We fix the scale $j_0 = \lfloor (1+ \beta_n ) j \rfloor = \lfloor (\frac{H}{\tau} + \frac{1}{n})j \rfloor$ that satisfies $2^{-\alpha j_0} \ge 2^{-\alpha (\frac{H}{\tau}+ \frac{1}{n}) j}$. By definition of $ND(j,\beta_n,\varepsilon_n)$, we know that for every $\lambda$ in this set, there are at least $2^{(\beta_n H  - 4 \max(1,\beta_n) \varepsilon_n)j}$ subintervals $\lambda'$ of $\lambda $ of scale $j_0$ such that $c_{\lambda'}=2^{-\alpha j_0}$. 
The independence of the Bernoulli random variables leads then to 
\begin{eqnarray*}
\mathbb{P} (\Omega_{j,n}) & \le & \sum_{\lambda \in ND(\beta_n,j)} \mathbb{P}\big(\forall \lambda_0 \subset \lambda {\text{ with } } \lambda_0 \in \G_{j_0}, \, \xi^{\eta}_{\lambda_0}= 0\big)\\
& \le & 2^{j(H+\varepsilon_n)} (1- 2^{(\eta-1) j_0})^{2^{(\beta_n H - 4 \max(1,\beta_n) \varepsilon_n)j}}\\
    & \le & 2^{j(H+\varepsilon_n)} \exp(-2^{(\beta_n H - 4 \max(1,\beta_n) \varepsilon_n)j} 2^{(\eta - 1) j_0})\\
    & \le & 2^{j(H+\varepsilon_n)} \exp (- 2^{j(\beta_n H - 4 \max(1,\beta_n)\varepsilon + (\eta - 1) (\beta_n + 1))} )\\
& =& 2^{j(H+\varepsilon_n)} \exp(-2^{j(\frac{1}{n} \tau - 4 \max(1,\beta_n) \varepsilon_n)})\\
& = & 2^{j(H+\varepsilon_n)} \exp(-2^{\frac{\tau}{2n}j}).
\end{eqnarray*}
Since $\tau>0$,  the Borel Cantelli Lemma implies the existence almost sure of a scale $J_n \in \mathbb{N}$ such that for all $j \ge J_n$  
and all $\lambda \in ND(j,\beta_n, \varepsilon)$, one has
$$
\sup_{\lambda' \subset \lambda} |\xi^{\eta}_{\lambda'}c_{\lambda'} | > 2^{-\alpha (1+ \beta_n) j}= 2^{-\alpha (\frac{H}{\tau}+\frac{1}{n}) j}.
$$
In particular, almost surely, one has 
$$
\# \big\{\lambda \in \Lambda_j : \sup_{\lambda' \subset \lambda} |\xi^{\eta}_{\lambda'}c_{\lambda'} |  \geq 2^{-\alpha (\frac{H}{\tau}+\frac{1}{n})j } \big\} \geq \# ND(j,\beta_n,\varepsilon_n) \geq 2^{(H-\varepsilon_n)j}
$$
for $j$ large enough.  Given any $\varepsilon>0$, if $n$ is such that $\frac{\alpha}{n}\leq \varepsilon$, we get that
$$
\# \big\{\lambda \in \Lambda_j : \sup_{\lambda' \subset \lambda} |\xi^{\eta}_{\lambda'}c_{\lambda'} |  \geq 2^{-(\alpha\frac{H}{\tau}+\varepsilon)j } \big\} \geq \# \big\{\lambda \in \Lambda_j : \sup_{\lambda' \subset \lambda} |\xi^{\eta}_{\lambda'}c_{\lambda'} |  \geq 2^{-\alpha (\frac{H}{\tau}+\frac{1}{n}) j} \big\} \geq 2^{(H-\varepsilon_n)j}
$$
hence 
$$
\rho_{cum, L_{\eta}(F^{(\alpha, (\I_j))})}\left(\frac{\alpha H}{\tau}\right) 
\geq \lim_{n \to + \infty} H-\varepsilon_n = H.
$$\hfill \qed

\bigskip

\noindent{\bf Proof of Part 3 of Theorem \ref{TheoDetect}} 
Let  $N_j$ denote the number of non-zero coefficients of $L_\eta (F^{(\alpha, (\I_j))})$ at scale $j$. Clearly, one has
$$\mathbb{E}[N_j] \leq  2^{(H+\varepsilon+ \eta - 1)j}$$ 
and an application of Chebychev inequality gives
$$
\sum_{j \in \N} \mathbb{P}\big(|N_j - \mathbb{E}[N_j]| \geq 2^{\varepsilon j } \big) \leq \sum_{j \in \N}\frac{2^{(H+\varepsilon+ \eta - 1)j}}{2^{2 \varepsilon j} } < + \infty
$$
since $H + \varepsilon + \eta -1 <0$. The Borel Cantelli Lemma implies then that  almost surely, 
 $$N_j \leq 2^{(H+ \varepsilon + \eta -1) j }  +  2^{\varepsilon j} \leq 1+ 2^{\varepsilon j}$$
 for every $j$ large enough.  Since $\varepsilon>0$ is arbitrary, it follows that 
 $$
\limsup_{j \to + \infty} \frac{\log \# N_j}{\log 2^j} \leq 0.
 $$
 One easily conclude that $\rho_{L_\eta F^{(\alpha, (\I_j))}}(\beta) = 0$ for every $\beta \geq \alpha$, since a wavelet leader of order $2^{-\beta j}$ comes from a wavelet coefficient equals to $2^{-\alpha j'}$ at scale $j' = \frac{\beta}{\alpha}j$. 
 \hfill \qed 

\medskip

Let us now turn to the proof of Corollary \ref{cor:validity}. Note that contrarily to the functions $F^{(\alpha,(\I_j))}$, we will not impose anymore an embedded structure for the non-zero coefficients. \\

\noindent \textbf{Proof of Corollary \ref{cor:validity}.}
Let us set $H={\mathcal{D}}_f(\alpha_0)=\rho_f(\alpha_0)$ and let us start by showing that
$$
{\mathcal{D}}_{S(f)} (\alpha ) = \rho_{S(f)}(\alpha ) = H.
$$
Recall that $S(f)$ is defined by \eqref{eq:sature}. First, it is clear that
$$
\{x : h_{{f}} (x) = \alpha_0 \} \subset  \{x : h_{S(f)} (x) = \alpha \},
$$
hence $\mathcal{D}_{S(f)} (\alpha ) \ge H $. Since $\mathcal{D}_{S(f)} (\alpha ) \leq \rho_{S(f)}(\alpha )$ by Proposition \ref{prop:comparison}, it remains to prove that $\rho_{S(f)}(\alpha ) \le H$.
This inequality is a consequence of the properties of the increasing wavelet leader profile \cite{Bastin:16}. Indeed, we have 
\begin{eqnarray*}
\rho_{S(f)}(\alpha ) & = & \limsup_{j \to + \infty} \frac{\log \#\{k : s_{j,k} \neq 0\}}{\log 2^j} \\
& \le & \limsup_{j \to + \infty} \frac{\log \#\{k : \sup_{\lambda' \subseteq \lambda_{j,k}}|c_{\lambda'}| \geq 2^{-j (h_{max}+1)}\}}{\log 2^j}\\
& =&  \rho_{cum,f}(h_{max}+1).
\end{eqnarray*}
Using \eqref{eq:rhocum}, we obtain
$$
\rho_{cum,f}(h)= \sup_{h' \le h} \rho_f(h') \le H,
$$
hence the announced upper bound.

To conclude, note now that the function $S(f)$ can  be written as a function $F^{(\alpha,(\I_j))}$ where 
$$
\I_j = \bigcup_{k\, : \, s_{j,k} = 2^{-\alpha j}} \big[2^{-j}k, 2^{-j}(k+1)\big].
$$
Hence, for every $\eta \in (1-H,1]$, the maximum of $\rho_{L_\eta S(f)}$ remains equal to $H$ by applying the second part of Theorem \ref{TheoDetect}. 
\hfill \qed \\

Finally, let us give an example which proves that this method cannot detect all failures in the multifractal formalism of wavelet series of the considered form. To this end,  we introduce two symmetric Cantor sets $\C_1 \subset [0,1/4]$ and $\C_2$ whose first step of the construction is $[3/4,1]$. We denote by $H_1$ and $H_1$ the Hausdorff dimension of $\C_1$ and $\C_2$ respectively, assuming that $H_1<H_2$. We set $J_1=2$ and $J_n = 2^{J_{n-1}}$. Finally, we define $f=f_1+f_2$ where
$$
f_1 = \sum_{j\in \N}\sum_{k=0}^{2^j-1}c^{(1)}_{j,k} \psi_{j,k} \quad \text{with} \quad c^{(1)}_{j,k}= \begin{cases} 2^{-\alpha j } &  \text{ if } \lambda_{j,k} \cap \C_1 \neq \emptyset \\ 
0 & \text{elsewhere}\\
\end{cases}
$$
and 
$$
f_2 = \sum_{j\in \N}\sum_{k=0}^{2^j-1}c^{(2)}_{j,k} \psi_{j,k} \quad \text{with} \quad c^{(2)}_{j,k}= \begin{cases} 2^{-\alpha j} &  \text{ if } \lambda_{j,k} \cap \C_2 \neq \emptyset \text{ and } \lambda \subset [1-2^{-J_n},1]  \text{ with } J_n \leq j < J_{n+1},\\ 
0 & \text{elsewhere.}
\end{cases}
$$
The function $f_1$ satisfies 
$${\mathcal{D}}
_{f_1}(\beta)= \rho_{f_1}(\beta)= \begin{cases} H_1 & \text{ if } \beta= \alpha ,\\
1 & \text{ if } \beta=+\infty, \\
- \infty & \text{ otherwise, }\end{cases}
$$
when $f_2$ satisfies
$$
\mathcal{D}_{f_2}(\beta) = \begin{cases} 0 & \text{ if } \beta= \alpha , \\ 1 & \text{ if } \beta=+\infty ,\\ - \infty & \text{ otherwise, } \end{cases}
$$
since $h_{f_2}(x) = +\infty$ for every $x \in [0,1[$ and $h_{f_2}(1)=\alpha$. An easy computation shows that the number of wavelet leaders $d^{(2)}_{j,k}$ equal to $ 2^{- \alpha j} $ is of order $ 2^{ H_2 j-J_n}$ if $J_n \le j <J_{n+1}$ (note that the coefficients are hierarchical and hence $d^{(2)}_{j,k}$ equals either $ 2^{-\alpha j}$ or $0$). It follows that  $\rho_f(\alpha)=H_2$ and $\mathcal{D}_f(\alpha) = H_1$, hence  $f$ does not satisfy the leader large deviation multifractal formalism at $\alpha$.

We consider the lacunarization of the process of parameter $\eta$. We have
$$
\EE \big[ \#\{ k\in \{0, \dots, 2^j-1\} :  \xi_{j,k}c^{(2)}_{j,k} \neq 0 \}\big] \simeq 2^{H_2j - J_n} 2^{(\eta-1)j} \simeq 2^{(H_2+\eta-1)j - \log_2 j}
$$
for $\frac{J_{n+1}}{2} \le j < J_{n+1} $, which leads to a constant maximum value for $\rho_{L_{\eta}(f_2)}$  equal to $H_2$, obtained for $\beta = \frac{\alpha}{H_2+\eta-1}$. It follows that the lacunarization operator does not does not reveal the overestimation of $\mathcal{D}_f$ via $\rho_f$.\\

\noindent \textbf{Acknowledgment.}  The authors thank the referees for their numerous remarks and suggestions, which greatly improved this text.

\bibliographystyle{plain}
\bibliography{ContrEx.bib}

\newpage

\appendix
\section{Lacunary wavelet series on Cantor sets : Proof of Proposition \ref{thm:LWS}}\label{sec:LWS}

Recall that the model is defined by the
random wavelet series $F_{\alpha, \eta ,r}=  \sum_{j \in \N}\sum_{k=0}^{2^j-1}c_{j,k}\psi_{j,k}$
with
$$
c_{j,k}= \begin{cases}
2^{-\alpha j} \xi_{j,k} & \text{ if }  k \in \Gamma_{j} \\
0 & \text{ otherwise,}
\end{cases}
$$
and
$$
\Gamma_{j} = \big\{ k  \in \{0, \dots, 2^j-1 \} : \lambda_{j,k}\subset  \C_{n_{j}}\big\}
\quad \text{ with } \quad n_{j}= \left\lfloor -
\frac{j}{\log_2 r}\right\rfloor= \left\lfloor \gamma j\right\rfloor,
$$
 where $\gamma = \dim_{\mathcal{H}}\mathcal{C}(r)$ and 
where  $ (\xi_{j,k})_{j,k}$ denotes a sequence of independent random Bernoulli
variables of parameter $2^{(\eta-\gamma)j}$.  By construction, the number of random wavelet coefficients at scale $j$
is of order $ 2^{ \gamma j }$. More precisely, if $j = \frac{n}{\gamma }+ b$ with $b \in [0,\frac{1}{\gamma})$, then the number of random coefficients at scale $j$ is given by
$
2^{n+b}= 2^{\lfloor \gamma j \rfloor+b}.
$
Consequently, one obtains
\begin{equation}\label{eq:espFjclassic}
 2^{\eta j -1}\leq \EE[\#\{\lambda \in \Lambda_{j} : c_{\lambda} = 2^{-\alpha j}\} ] \leq  2^{\lfloor \gamma j \rfloor+b} 2^{(\eta-\gamma)j} \leq 2^{\eta j + \frac{1}{\gamma}}.
\end{equation}

Let us start by studying the maximal regularity of the lacunary
wavelet series $F_{\alpha, \eta ,r}$. Clearly, if $x \notin \C(r)$, then $3 \lambda_{j}(x)
\cap \C(r) = \emptyset$ for $j$ large enough and $h_{f}(x) = +
\infty$. 

\begin{lemma}\label{lem:regmaxclassic}
    Almost surely, there is $J \in \N$ such
that
$$
d_\lambda \geq \sup_{\lambda' \subset  \lambda}|c_{\lambda'}| \geq 2^{-\frac{\alpha
  }{\eta}(\gamma j + \log_{2} \gamma j)} 
$$
for every $\lambda \in \Gamma_{j}$ with $j \geq J$. In particular, $h_{f}(x) \leq \frac{\alpha
  \gamma }{\eta}$ for every $x \in \C(r)$.
\end{lemma}

\begin{proof}
 For every $j \ge 0$, let us consider the event
$$
\Omega_j = \{ \exists \lambda \in \Gamma_{j} \, \text{ such that
}\sup_{\lambda' \subset  \lambda} |c_{\lambda'}|  < 2^{-\frac{\alpha
  }{\eta}(\gamma j + \log_{2} \gamma j)} \}.
$$
We fix the scale  $j_0 =\lfloor \frac{1}{\eta} (\gamma j + \log_2 \gamma 
j)\rfloor$  that satisfies 
$2^{-\alpha j_0} \ge 2^{-\frac{\alpha}{\eta}(\gamma j + \log_{2}
  \gamma j)}$. By the independence of the Bernoulli random variables  and since there is about $2^{\gamma(j_0-j)}$ dyadic intervals in $\Gamma_{j_0}$ inside a dyadic interval $\lambda \in \Gamma_j$, we easily obtain   that 
$$\PP(\Omega_{j})
 \leq  \sum_{\lambda \in \Gamma_{j}} \PP(\forall \lambda_0 \subset  \lambda \text{ with } \lambda \in \Gamma_{j_0} , \xi_{\lambda_{j_0}}=0) \leq C\left(\frac{2}{e}\right)^{\gamma j}
$$
for some positive constant $C$ and $j$ large enough. The conclusion follows from the
Borel-Cantelli lemma and Theorem  \ref{thm:waveletcharact}.
\end{proof}

Hence, the range for  the possible values of the H\"older exponent of
points belonging to $\C(r)$ is $[\alpha, \frac{\alpha\gamma}{\eta}]$.
Let us now describe the iso-H\"older sets of $f$. Let us start by
giving a covering of $\C(r)$ using balls centered at the dyadic points
associated with non-zero coefficients. For this purpose, let us
introduce for each scale $j$ the random set $A_{j}$ defined by
$$
A_{j} = \big\{ k \in \{0, \dots, 2^{j}-1\} : c_{j,k} = 2^{-\alpha j}\big\} .
$$
We proceed as in the proof of Proposition \ref{prop:recouvrementK} using Lemma \ref{lem:regmaxclassic} to get the following covering of $\C(r)$.

\begin{corollary}\label{cor:inclusionCantorclassic}
  Almost surely, one has
$$
\C(r) \subset  
 \limsup_{j \to + \infty} \bigcup_{k \in A_{j}} B\left(k2^{-j},
  2^{-\frac{\eta}{\gamma}(1-{\varepsilon_{j}})j}\right) 
\quad \text{where}
\quad {\varepsilon_{j}}= \frac{\log_{2} \gamma j}{\eta j}.$$
\end{corollary}

As done for the duplicated lacunary wavelet series, for every $\delta
\in (0,1]$, we consider the random set
\[
E_{\delta}:= \limsup_{j \rightarrow + \infty}\bigcup_{k \in A_{j}
}B\left( k 2^{-j}, 2^{-\delta (1-\varepsilon_{j})j} \right).
\]
Then, we set
\[
G_{\delta } := \bigcap_{0<\delta' < \delta} E_{\delta'} \setminus \bigcup_{\delta < \delta' \leq 1} E_{ \delta'} \ \text{ if } \ \delta <1  \quad \text{ and }   \quad G_1 := \bigcap_{0<\delta' < 1} E_{\delta'}.
\]
Using Remark \ref{rem:lemmaEdelta}, since the points lying outside $\C(r)$ have an infinite H\"older
exponent, we have
\begin{equation}\label{eq:Gdeltaclassic}
  G_{\delta} = \left\{x \in [0,1] : h_{F_{\alpha,\eta, r}}(x) = \frac{\alpha}{\delta}\right\}.
\end{equation}
Consequently, in order to compute the multifractal spectrum of $F_{\alpha,\eta, r}$,
it suffices to study the Hausdorff dimension of the sets
$G_{\delta}$. We also already know that we can restrict ourselves to
the values of $\delta$ belonging to $[\frac{\eta}{\gamma},1]$.  Using \eqref{eq:espFjclassic}, we can estimate the cardinality of $A_{j}$. This estimation then provides an upper bound for $\dim_{\mathcal{H}}G_{\delta}$.

\begin{proposition}\label{prop:upperboundclassic}
 Almost surely, 
 \begin{itemize}
     \item for every $\varepsilon >0$, there is $J \in \N$ such
that
$$
2^{( \eta - \varepsilon) j} \le \# A_{j} \leq 2^{( \eta + \varepsilon) j}  \quad \forall j \geq J.
$$
     \item for every $\delta \in [\frac{\eta}{\gamma},1]$, one
 has
 $$
 \dim_{\mathcal{H}}(G_{\delta}) \leq \frac{\eta}{\delta}
 $$
and  $\mathcal{H}^{\eta/\delta}(E_{\delta'}) = 0$ for all
 $\delta'>\delta$. 
 \end{itemize}
\end{proposition}

\begin{proof}
The first part follows directly from \eqref{eq:espFjclassic} combined with Chebyshev's inequality and the Borel-Cantelli lemma.

For the second part, we fix $\delta '<\delta$ and we consider the set $E_{\delta'}$ as a covering of $G_{\delta}$. The first part of the proof implies that almost surely, for every $\varepsilon>0$,
there is $J \in \N$ such that 
$$\sum_{j \geq J} \sum_{k \in F_{j}} 2^{-\delta' (1-\varepsilon_{j})s j}
 \leq  \sum_{j \geq J}2^{(\eta + \varepsilon -\delta'
   (1-\varepsilon_{j})s )j} <+\infty
$$
if $s > \frac{ \eta + \varepsilon}{\delta'(1-\varepsilon)}$, since
$\varepsilon_{j} \leq \varepsilon$ for $j$ large enough. Hence
$
\mathcal{H}^s (G_{\delta})  < + \infty
$
and therefore, $\dim_{\mathcal{H}}(G_{\delta}) \leq s$.

Finally, note  that ${E}_{\delta'} \subset 
\bigcap_{0<\delta''< \delta'}E_{\delta''}$. By proceeding as
previously, one gets that $\dim_{\mathcal{H}}E_{\delta'} \leq
\frac{\eta}{\delta'}$. The conclusion follows easily. 
\end{proof}

\begin{proposition}\label{prop:lowerboundclassic}
With probability one, for every $\delta \in [\frac{\eta}{\gamma},1]$, $\dim_{\mathcal{H}}(G_{\delta}) \geq  \frac{\eta}{\delta}$.
\end{proposition}

\begin{proof}
From Corollary \ref{cor:inclusionCantorclassic},
we know that almost surely, $$
\C(r) \subset  
 \limsup_{j \to + \infty} \bigcup_{k \in F_{j}} B\left(k2^{-j},
  2^{-\frac{\eta}{\gamma}(1-{\varepsilon_{j}})j}\right).
$$
By multiplying the radius of the balls by a constant independent of
$j$, we may moreover assume that the balls are centered at points of the
Cantor set $\C(r)$. Note that $\C(r)$ is a Cantor set  and then it satisfies assumption (\ref{eqGMT}) of the General mass transference principle given in Theorem
\ref{thm_transference}, see Theorem 4.14 of \cite{Mattila}. Consequently, this principle gives
\[
\mathcal{H}^{\eta/\delta} \left(  \C(r) \cap 
 \limsup_{j \to + \infty} \bigcup_{k \in F_{j}} B\left(k2^{-j},
  2^{-\delta(1-{\varepsilon_{j}})j}\right) \right)  =
\mathcal{H}^{\eta/\delta}(\C(r)). 
\]
Since $\eta/\delta \leq \gamma$, we obtain
$\mathcal{H}^{\eta/\delta}(E_{\delta}) > 0 $ and $\dim_{\mathcal{H}}(E_{\delta}) \ge \eta/\delta$. The computation of the lower bound of $\dim_{\mathcal{H}}(G_\delta)$ is then obtained as in the proof of Proposition \ref{prop_sup}.
\end{proof}

In order to get Proposition \ref{thm:LWS}, it remains now to compute the leader large deviation spectrum of $F_{\alpha, \eta, r}$.

\begin{proposition}
Almost surely, $\mathcal{D}_{F_{\alpha, \eta,r}}(h)= \rho_{F_{\alpha, \eta,r}}(h)$ for every $h \in [0,
+ \infty]$. 
\end{proposition}

\begin{proof}
From the construction of the lacunary wavelet series $F_{\alpha, \eta, r}$ and using Lemma \ref{lem:regmaxclassic}, it is
clear that
$\rho_{F_{\alpha, \eta,r}}(h)=- \infty$ if $h \notin [\alpha, \frac{\alpha \gamma
}{\eta}] \cup \{+ \infty\}$ and that $\rho_{F_{\alpha, \eta,r}}(+ \infty) = 1$. So, let
$h \in [\alpha, \frac{\alpha \gamma
}{\eta}]$. Then, using the first part of Proposition \ref{prop:upperboundclassic}, one has
$$\#\{ \lambda \in \Lambda_{j } : 2^{- (h+\varepsilon)j} \leq
e_{\lambda} \leq 2^{- (h-\varepsilon)j} \} \leq \sum_{j' = \lfloor
  \frac{h-\varepsilon}{\alpha} j\rfloor}^{\lfloor
  \frac{h+\varepsilon}{\alpha} j\rfloor +1} \#A_{j'} \leq C j 2^{(\eta
  + \varepsilon)\frac{h+\varepsilon}{\alpha} j}$$
for some constant $C>0$ and $j$ large enough. The upper bound for
$\rho_{F_{\alpha, \eta, r}}(h)$ follows directly. The lower bound is given by the general
inequality $\mathcal{D}_{F_{\alpha, \eta, r}} \leq \rho_{F_{\alpha, \eta, r}}$. 
\end{proof}

\end{document}